\documentclass[psamsfonts]{amsart}

\usepackage{amssymb,amsfonts,amssymb,amsthm,mathabx}
\usepackage[all,arc]{xy}
\usepackage{enumerate}
\usepackage{mathrsfs}
\usepackage{hyperref}
\usepackage{dsfont}
\usepackage{bm}
\usepackage{pgfplots}
\usepackage[T1]{fontenc}
\allowdisplaybreaks

\usepackage{amsrefs}

\newtheorem{theorem}{Theorem}[section]
\newtheorem{corollary}[theorem]{Corollary}
\newtheorem{proposition}[theorem]{Proposition}
\newtheorem{lemma}[theorem]{Lemma}

\newtheorem{remark}[theorem]{Remark}

\newtheorem{definition}[theorem]{Definition}

\newcommand{\R}{\mathbb{R}}
\newcommand{\C}{\mathbb{C}}

\newcommand{\N}{\mathbb{N}}

\newcommand{\suma}[1]{\sum\limits_{#1}}

\newcommand*{\bigchi}{\mbox{\large$\chi$}}

\newcommand{\la}{\langle}
\newcommand{\ra}{\rangle}
\newcommand{\Lp}{L_p}
\newcommand{\Lone}{L_1}

\newcommand{\Linfty}{L_\infty}

\newcommand{\Ltwo}{L_2}
\newcommand{\Lq}{L_q}
\newcommand{\Lphalf}{L_{p/2}}
\newcommand{\Lpprime}{L_{p^\prime}}
\newcommand{\Ltwozero}{L_2^\circ}

\newcommand{\Ltwoc}{L_2^c}
\newcommand{\Ltwor}{L_2^r}

\newcommand{\Ltwoczero}{L^{\circ,c}_2}

\newcommand{\Hone}{\mathrm{H}_1}
\newcommand{\Honec}{\mathrm{H}_1^c}
\newcommand{\Honer}{\mathrm{H}_1^r}
\newcommand{\Hc}{H^c}

\newcommand{\Kc}{K^c}

\newcommand{\BMOc}{\mathrm{BMO}^c}
\newcommand{\BMOr}{\mathrm{BMO}^r}
\newcommand{\BMO}{\mathrm{BMO}}
\newcommand{\Hardyc}{\Honec(\mathcal{A})}
\newcommand{\Hardyr}{\Honer(\mathcal{A})}
\newcommand{\Hardy}{\Hone(\mathcal{A})}
\newcommand{\id}{\mathrm{Id}}

\newcommand{\Tw}{\widetilde{T}_\omega}

\newcommand{\Avnat}{\Linfty(\mathbb{R}) \overline{\otimes} \mathcal{M}}

\title{Calder\'on-Zygmund theory with noncommuting kernels via $H_1^c$}

\author[A.I. Cano-M\'armol]{Antonio Ismael Cano-M\'armol}
\address{Department of Mathematics \\ Baylor University \\
1301 S University Parks Dr\\
Waco, TX 76798, USA.}
\email{AntonioIsmael\_CanoMa@baylor.edu}

\author[\'E. Ricard]{\'Eric Ricard}
\address{Laboratoire de Math\'ematiques Nicolas Oresme \\ UNICAEN, CRNS \\
1400 Caen, France.}
\email{eric.ricard@unicaen.fr}

\begin{document}


    
    \maketitle


    \begin{abstract}
    We study an alternative definition of the $\Hone$-space associated to a semicommutative von Neumann algebra $L_\infty(\mathbb{R}) \overline{\otimes} \mathcal{M}$, first studied by Mei. We identify a ``new'' description for atoms in $\Hone$. We then explain how they can be used to study $\Hone^c$-$\Lone$ endpoint estimates for Calder\'on-Zygmund operators with noncommuting kernels.
\end{abstract}



\maketitle

\section*{Introduction}

    This paper is related to the theory of semicommutative
    Calder\'on-Zyg\-mund operators. This is a research line that takes
    advantage of the hybrid nature of certain vector valued $\Lp$-spaces. Let $(\mathcal{M},\tau)$ be a von Neumann algebra of
    operators on a separable Hilbert space, equipped with a normal
    semifinite faithful trace $\tau$. Denote by
    $\mathcal{A}$
    the weak operator closure of the space of essentially bounded
    (strongly measurable) functions
    $f : \mathbb{R} \longrightarrow \mathcal{M}$ acting on $L_2(\mathbb R ; L_2(\mathcal M))$. The
    von Neumann algebra $\mathcal{A}$ can be identified with the
    tensor product
    $\Linfty(\mathbb{R}) \overline{\otimes} \mathcal{M}$ equipped with
    the trace
    \begin{align*}
        \varphi(f) = \int_{\mathbb{R}} \tau(f(x)) \ dx.
    \end{align*}  
    For the sake of
    exposition, we will restrict ourselves to dimension 1, even
    though our arguments extend trivially to any finite dimension namely for $\Linfty(\mathbb R^n) \overline \otimes \mathcal M$.

 The noncommutative $\Lp$-spaces associated with $\mathcal{A}$ are indeed vector-valued $\Lp$-spaces: more clearly \cite[Chapter 3]{P98}
$$
\Lp(\mathcal{A}) = \Lp(\mathbb{R};\Lp(\mathcal{M})),
$$    
for $1\leq p<\infty$. We are interested in endpoint estimates for
operators acting on $\Lp(\mathcal{A})$, and in particular in the
boundedness of operators from the operator-valued version of the Hardy
space $\Hone$ into $\Lone$. This question was widely studied in the
classical setting for scalar-valued functions \cite{Mey90,MeyC97} as
well as for vector-valued functions \cite{F90,H06}, where the
existence of the atomic decomposition plays an essential role. This
technique does not seem to have been exploited as often in the
noncommutative setting except perhaps in \cite{HLMP14}. Mei \cite{M07} was the first to introduce the
so-called \emph{operator-valued Hardy space}
$\Hone(\mathbb{R},\mathcal{M})$ in this context via noncommutative
equivalents of the Poisson integral, the Lusin area integral and the
Littlewood-Paley $g$ function. These techniques allowed Mei to identify
the dual space of $\Hone(\mathbb{R},\mathcal{M})$, which is denoted by
$\BMO(\mathbb{R},\mathcal{M})$, in the spirit of the classical
argument by Fefferman and Stein \cite{FS72}. Moreover, some maximal
inequalities, and several interpolation results via a martingale
approach were established. Mei's fundamental contribution has been a key
in the development of noncommutative forms of Calder\'on-Zygmund
theory, both in the semicommutative context and in fully
noncommutative ones via transference techniques. For the first one,
the semicommutative Calder\'on-Zygmund theory was initiated in
\cite{P09} with the first weak-$\Lone$ endpoint inequalities
for singular integrals; the arguments  were simplified in recent
years \cite{C18,CCP22}. The second line has many instances, among
which are \cite{GJP21,JMPX21}.
    
The initial motivation for the present work was obtaining new
interpolation consequences of endpoint estimates of the type
$\Linfty(\mathcal{A})$-$\BMO(\mathbb{R},\mathcal{M})$ which rely, by
duality, on the structure of the Hardy space
$\Hone(\mathbb{R},\mathcal{M})$. Our goal led to two main tasks to
tackle: a completely explicit description of
$\BMO(\mathbb{R},\mathcal{M})$, and the study of the boundedness of
Calder\'on-Zygmund operators on the Hardy space via its atomic
decomposition. This may seem common but another of our goals  is to provide complete proofs; for instance, in many papers the extension from atoms to the whole Hardy space is overlooked but requires more delicate arguments than the main estimates  (see \cite{B05}). We also want to point out that up to now there is no notion of noncommutative distributions, thus one has to be careful when defining Calder\'on-Zygmund operators.

The operator-valued $\BMO$-space introduced by
Mei, $\BMO(\mathbb{R},\mathcal{M})$, is defined as the
intersection of a column space and a row space,
$\BMOc(\mathbb{R},\mathcal{M})$ and $\BMOr(\mathbb{R},\mathcal{M})$
respectively. The reason for considering both a column and a row space
is a ubiquitous phenomenon in noncommutative analysis (see \cite{LP86}
for an outstanding example), and by symmetry we shall limit our
discussion to the column case. $\BMOc(\mathbb{R},\mathcal{M})$ is set
to be the subspace of the column Hilbert-valued space \cite{JLX06}
$\Linfty(\mathcal{M};\Ltwoc(\mathbb{R},\frac{dt}{1+t^2}))$ on  which
the seminorm
\begin{align}\label{eq:introBMOcnorm}\tag{MBMO}
        \|g\|_{\BMOc} = \sup_{\substack{I\subseteq \mathbb{R} \\ |I| \ \mathrm{finite}}}\Big\| \Big( \frac{1}{|I|} \int_I \big| g - g_I \big|^2 \Big)^{1/2} \Big\|_{\mathcal{M}}
    \end{align}
makes sense, where $g_I = \frac{1}{|I|} \int_I g$. The $\BMOr$ seminorm is $\|g\|_{\BMOr} = \|g^*\|_{\BMOc}$. The space $\Linfty(\mathcal{M};\Ltwoc(\mathbb{R},\frac{dt}{1+t^2}))$ denotes the closure of $\mathcal{M} \otimes \Ltwo(\mathbb{R},\frac{dt}{1+t^2})$ with respect to the weak$^*$ topology of the von Neumann algebra $\mathcal{M} \overline{\otimes} B(\Ltwo(\mathbb{R},\frac{dt}{1+t^2}))$. There is no guarantee that $\BMOc(\mathbb{R},\mathcal{M})$ is a space of $\mathcal{M}$-valued functions (in the Bochner sense), and so the integral in \eqref{eq:introBMOcnorm} may not be well-defined. Indeed,
$$
\Ltwo\big(\mathbb{R},\frac{dt}{1+t^2};\mathcal{M}\big) \subset \Linfty\big(\mathcal{M};\Ltwoc(\mathbb{R},\frac{dt}{1+t^2})\big),
$$
but the reverse inclusion fails in general, \emph{a priori} preventing us from defining $\BMOc$ as a space of functions. In section 1 and the first part of section 2, we propose a general construction of $\BMOc(\mathbb{R},\mathcal{M})$ which recovers Mei's description given by \eqref{eq:introBMOcnorm}. We will also study a predual of $\BMOc(\mathbb{R},\mathcal{M})$ (resp. $\BMOr(\mathbb{R},\mathcal{M})$). The novelty here is that this predual space, which we denote as $\Hardyr$ (resp. $\Hardyc$), will be a row (resp. column) Hardy space which is exclusively constructed in terms of ``new'' atomic decompositions, which extend the work of the second author's PhD Thesis \cite{R01}.

    The key to our approach is the $\Hone$-$\BMO$ duality product when elements in the former space are described in terms of atoms. In the classical case, it is a well-known fact that the norm of $g\in \BMO(\mathbb{R})$ can be characterized through the expression
    \begin{align}\label{eq:introBMOnormatoms}\tag{atBMO}
        \|g\|_{\BMO} = \sup_{a} \big| \int g a \big|,
    \end{align}
    so that the supremum is taken over $\Ltwo$-atoms \cite{GR85,Mey90,MeyC97}. An analogous formula for $\BMOr(\mathbb{R},\mathcal{M})$ may shed light on the structure of atoms in $\Hardyc$.  This is exactly what we achieve. First, it turns out that it is enough to consider only elements $g\in\mathcal{A} \cap \BMOr(\mathbb{R},\mathcal{M})$. The expression \eqref{eq:introBMOcnorm} is meaningful for $g$, and so duality yields
    \begin{align*}
        \|g\|_{\BMOr} &= \|g^*\|_{\BMOc} = \sup_{I} \Big\| \Big( \frac{1}{|I|} \int_I \big| g^* - (g^*)_I \big|^2 \Big)^{1/2} \Big\|_{\mathcal{M}} 
        \\ & = \sup_{I,\|h\|_{\Ltwo(\mathcal M)}\leq 1} \Big( \int_I \|\frac{1}{\sqrt{|I|}} h (g - g_I)\|^2_{\Ltwo(\mathcal{M})} \Big)^{1/2}.
    \end{align*}
    with the supremum taken over $h$ in the unit ball of $\Ltwo(\mathcal{M})$. Now, recalling that $g - g_I$ has zero integral over $I$, it follows that
    $$
        \|g\|_{\BMOr} = \sup_{I,h} \big\| \frac{1}{\sqrt{|I|}} h (g - g_I) \bigchi_I \big\|_{\Ltwozero(\mathbb{R};\Ltwo(\mathcal{M})} = \sup_{I,h} \sup_{\|f\|_{\Ltwo} \leq 1} \big| (\tau \circ \int)(h (g-g_I) \frac{f \bigchi_I}{\sqrt{|I|}})\big|.$$
    Comparing the latter expression with \eqref{eq:introBMOnormatoms} suggests that an atom in $\Honec(\mathbb{R},\mathcal{M})$ should be an operator of the form $a = b h$ in $\Lone(\mathcal{A})$,  where $h \in \Ltwo(\mathcal{M})$ with $\|h\|_{\Ltwo(\mathcal{M})} \leq 1$ and $b \in \Ltwo(\mathcal{A})$ is supported on some interval $I$ and has an additional cancellation over $I$ that we will make precise later.
In  what follows, $a$ will be called a \emph{$c$-atom}. Then, define the column Hardy space $\Hardyc$ as the Banach subspace of $\Lone(\mathcal{A})$ of those operators $f$ such that
    \begin{align*}
        f = \sum_{i=0}^\infty \lambda_i a_i \ \mathrm{in} \ \Lone(\mathcal{A}), \ \mathrm{for \ some} \ (a_i)_i \ c\textrm{-atoms}, \ (\lambda_i)_i \in \ell_1,
    \end{align*}
    which becomes a Banach space with respect to the norm
    \begin{align*}
        \|f\|_{\Honec} = \inf\big\{ \sum_{i=0}^\infty |\lambda_i| \ : \ f = \sum_{i=0}^\infty \lambda_i a_i \big\}.
    \end{align*}
    The row space $\Hardyr$ is defined analogously and $\Hardy = \Hardyc + \Hardyr$. By symmetry, it suffices to show $\Hardyc^* = \BMOr(\mathbb{R},\mathcal{M})$. The proof of this duality result strongly relies on the extension of a well-known argument by Meyer \cite{Mey90}: the space
    \begin{align*}
        \Ltwozero(\mathbb{R},(1+t^2)dt) = \big\{ f \in \Ltwo(\mathbb{R},(1+t^2)dt) \ : \ \int_{\mathbb{R}} f = 0 \big\}
    \end{align*}
    is a dense subspace of the classical atomic Hardy space
    $\Hone(\mathbb{R})$. A further characterization of the column
    Hilbert-valued spaces $\Linfty(\mathcal{M};\Hc)$ will be the key
    to establish an analogous result in our context. We will briefly
    explain how to get that the column Hardy space $\Hardyc$ coincides with the
    one introduced by Mei in \cite{M07} in section \ref{example}. This
    will in particular allow us to indeed use interpolation as usual.

    It is worth noting that Mei's work already contains a description
    of $\Honec(\mathbb{R},\mathcal{M})$ in terms of certain atomic
    decompositions. However, the one included in the present paper
    seems more useful to deal with Calder\'on-Zygmund operators and to
    make connections with vector valued harmonic analysis. Indeed, it
    allows us to consider singular integrals with noncommuting
    kernels, something that looks difficult to reach at the
    weak-$\Lone$ level \cite{CCP22}.  Let $\mathcal{M}$
      be a von Neumann algebra over a separable Hilbert space, we
      denote by $\mathcal S$ the set of compactly supported  essentially
      bounded functions
      $\mathbb R\to \Linfty \cap \Lone(\mathcal{M})$ (measurable with values in $L_1$); it
      will play the role of test functions.
Let $T$ be a
    bounded operator on $\Ltwo(\mathcal{A})$ for which there exists a
    kernel
    $K : \mathbb{R} \times \mathbb{R} \ \setminus \ \{x=y\}
    \longrightarrow \mathcal{M}$ such that
    \begin{align*}
        \int T(f)(x) g(x) \ dx = \int \int K(x,y) f(y) g(x) \ dx \ dy
    \end{align*}
    holds for any $f,g \in \mathcal S$ satisfying that the distance between the supports $\mathrm{supp}\|f\|_{\Ltwo(\mathcal{M})}$ and $\mathrm{supp}\|g\|_{\Ltwo(\mathcal{M})}$ is strictly positive. In that case, We say that $T$ is a Calder\'on-Zygmund operator with kernel $K$. Also, assume that $T$ fulfills
    \begin{itemize}
        \item a right-modularity condition: $T(fh) = T(f)h$ for any $f \in \Ltwo(\mathcal{A})$ with compact support and $h \in \mathcal{M}$; we say that $T$ is a left Calder\'on-Zygmund operator,
        \item the \emph{Hörmander condition} for some $\lambda>0$
        \begin{align*}
            \int_{|x-y| \geq \lambda |y^\prime -y|} \|K(x,y)-K(x,y^\prime)\|_{\mathcal{M}} \ dx < \infty.
        \end{align*}        
      \end{itemize}
      
        Under these assumptions, $T$ extends to a bounded map $T : \Hardyc \longrightarrow \Lone(\mathcal{A})$. The proof of this statement is inspired by \cite{MeyC97} and is divided in two steps. The first one consists in obtaining a  constant $C_\lambda > 0$ such that
    \begin{align*}
        \|T(a)\|_{\Lone(\mathcal{A})} \leq C_\lambda \ \textrm{for \ any \  $c$-atom} \ a.
    \end{align*}
    
    In order to set up this statement, we first check that the extension of $T$ to $c$-atoms satisfies
    \begin{align*}
        T(a) = T(b)h \ \textrm{for \ any \ $c$-atom}\  a = bh.
    \end{align*}
    This allows us to exploit the boundedness of $T$ on $\Ltwo(\mathcal{A})$.
    We rely on
    an approximation of $T$ and $K$ by certain  uniformly bounded kernels, which also implies that $T$ extends to the whole $\Hardyc$. Once we have done that, we obtain that whenever $K$ is scalar and modularity works on both sides,  then $T$ extends to a bounded operator from $\Hardy$ into $\Lone(\mathcal{A})$. This is  somehow a dual statement of a result in \cite{JMP14} where $L_\infty$-BMO estimates are given and  in spirit close to \cite{HLMP14}.

    The rest of the paper is organized as follows. In section 1, we
    shall introduce the column and row Hilbert-valued noncommutative
    $\Lp$-spaces. This enables us to define
    $\BMO(\mathbb{R,\mathcal{M}})$, as well as to identify a
    predual $\Hardyc$ in section 2. In section 3 we will see how the
    atomic decomposition provides a boundedness result for
    Calder\'on-Zygmund operators with noncommuting kernels. We then
    briefly explain in section 4 how to connect our construction
    to Mei's one. We end with some technical but useful lemmas in the appendices.

    \section{Column/row Hilbert-valued $\Lp$-spaces}

  In this section we collect some basic facts on classical vector-valued $L_p$-spaces \cite{HNVW16}, as well as its interaction with row and column Hilbert-valued $L_p$-spaces \cite{JLX06}.

Before giving definitions, some facts about vector-valued functions and von Neumann algebras should be tackled (see \cite[Chapter 2]{DU77} and \cite[Section 1.22]{S71}).  Let $\mathbb{X}$ be a Banach space, and let $(\Omega,\mu)$ be a $\sigma$-finite measure space (or more generally, a localizable measure space \cite{S73}). Then, a function $f : \Omega \longrightarrow \mathbb{X}$ is said to be \emph{$\mu$-measurable} whenever there exists a sequence of simple functions $(f_n)_{n \geq 1}$, $f_n = \sum_{i} x_i^n \ \bigchi_{A_i^n}$ for some $x_i^n \in \mathbb{X}$ and $\mu-$measurable sets $A_i^n$, such that
    \begin{align*}
        \lim_{n \rightarrow \infty} \|f_n - f \|_{\mathbb X} = 0 \ \mu\mathrm{-almost \ everywhere}.
    \end{align*}
    For $1\leq p<\infty$, one can then define
    $\Lp(\Omega,\mu;\mathbb{X})$ as the space of $\mu$-measurable
    functions $f$ for which $\|f(\cdot)\|_{\mathbb{X}}\in \Lp(\Omega,\mu)$
    with the norm
    $\|f\|_p=\| \|f(\cdot)\|_{\mathbb{X}}\|_{\Lp(\Omega,\mu)}$.
    In our situation, we recall the classical identities:
    \begin{align*}
      \Lone(\Omega,\mu;\Lone(\mathcal M))=\Lone(L_\infty(\Omega,\mu)\overline \otimes \mathcal M)=  \Lone(\Omega,\mu)\widehat \otimes_\pi \Lone(\mathcal M);\\
      \Ltwo(\Omega,\mu;\Ltwo(\mathcal M))=\Ltwo(L_\infty(\Omega,\mu)\overline \otimes \mathcal M)=\Ltwo(\Omega,\mu) \otimes_2 \Ltwo(\mathcal M),
    \end{align*}
    where $\widehat \otimes_\pi$ and $\otimes_2$ are the Banach space
    projective tensor product and the Hilbert space tensor product. On
    the other hand, the identification of the von Neumann algebra
    $\Linfty(\Omega,\mu) \overline{\otimes} \mathcal{M}$ as a space of
    vector-valued functions requires a more subtle construction. In
    general, a function $f : \Omega \longrightarrow \mathbb{X}^*$ is
    said to be \emph{weak$^*$ $\mu$-measurable} if $J_x \circ f$ is
    measurable for each $x \in \mathbb{X}$, where $J_x$ denotes the
    continuous functional on $\mathbb{X}^*$ given by
    $J_x(x^*) = x^*(x)$ for every $x^*$ in $\mathbb{X}^*$. Consider
    $\mathbb{X} = \Lone(\mathcal{M})$  (which is assumed to be separable to avoid measurability issues) and define
    $\Linfty(\Omega,\mu;\mathcal{M})$ as the Banach space of all
    $\mathcal{M}$-valued weak$^*$ $\mu-$measurable functions which are
    essentially bounded, that is,
    \begin{align*}
        {\mathrm{ess}\sup}_{t \in \Omega} \|f(t)\|_{\mathcal{M}} < \infty.
    \end{align*}
    Then, $\Linfty(\Omega,\mu;\mathcal{M})$ is a von Neumann algebra under the pointwise multiplication and the map
    \begin{align*}
        f \otimes m \mapsto f(t)m, \quad f \in \Linfty(\Omega), \ m \in \mathcal{M},
    \end{align*}
    can be extended to a weak$^*$ isomorphism of $\Linfty(\Omega,\mu) \overline{\otimes} \mathcal{M}$ onto $\Linfty(\Omega, \mu;\mathcal{M})$.

    \medskip

    Let $H$ be a separable
    Hilbert space. $B(H)$ can be identified as the space of bounded
    infinite matrices acting on $H$ (once a basis is fixed). When equipped with the usual
    trace for matrices $\mathrm{Tr}$, it gives rise to the Schatten
    classes $S_p(H) = \Lp(B(H),\mathrm{Tr})$ for any
    $0 < p \leq \infty$. Along this work, the inner product in $H$ is
    assumed to be linear in the second variable and antilinear in the
    first one. Moreover, elements of $H$ viewed in the dual Hilbert
    space $H^*$ will be represented with an overlined letter. We mean
    that $\overline{h}$ will denote the continuous functional
    \begin{align*}
        \overline{h} : k \in H \mapsto \langle h, k \rangle \ \mathrm{for \ any \ } k \in H.
    \end{align*}
    Given two elements $\xi$ and $\eta$ of $H$, we consider
    the \emph{rank-one operator} $\xi \otimes \overline \eta$ acting on $H$ as
    follows
    \begin{align*}
        (\xi \otimes \overline \eta)(h) = \langle \eta, h \rangle \ \xi \ \mathrm{for \ any \ } h \in H.
    \end{align*}

    In the following, let $\mathds{1}$ denote a fixed element of $H$ with $\|\mathds{1}\|_H = 1$, and let $p_{\mathds{1}} = \mathds{1} \otimes \overline{\mathds{1}}$ denote the rank one projection onto $\mathrm{span}\{\mathds{1}\}$. Assume that $\mathcal{M}$ is an arbitrary semifinite von Neumann algebra equipped with a normal semifinite faithful trace $\tau$. Then, we define the \emph{column Hilbert-valued $\Lp$-space}
    \begin{align*}
        \Lp(\mathcal{M};\Hc) = \Lp(\mathcal{M}\overline{\otimes} B(H))(\mathbf{1}_{\mathcal{M}}\otimes p_{\mathds{1}}).
    \end{align*}
    for any $0 < p \leq \infty$. Identify $\Lp(\mathcal{M})$ as a subspace of $\Lp(\mathcal{M} \overline{\otimes} B(H))$ via the map  $m \mapsto m \otimes p_{\mathds{1}} = ( \mathbf{1}_\mathcal{M} \otimes p_{\mathds{1}}) (m \otimes \mathbf{1}_{B(H)}) ( \mathbf{1}_\mathcal{M} \otimes p_{\mathds{1}})$. This is equivalent to the identity
    \begin{align*}
        \Lp(\mathcal{M}) = (\mathbf{1}_{\mathcal{M}} \otimes p_{\mathds{1}}) \Lp(\mathcal{M} \overline{\otimes} B(H)) (\mathbf{1}_{\mathcal{M}} \otimes p_{\mathds{1}}).
    \end{align*}
    Then, given an element $f$ in $\Lp(\mathcal{M};\Hc)$,
    $$
    f^*f \in (\mathbf{1}_{\mathcal{M}} \otimes p_{\mathds{1}}) \Lphalf(\mathcal{M} \overline{\otimes} B(H)) (\mathbf{1}_{\mathcal{M}} \otimes p_{\mathds{1}}) = \Lphalf(\mathcal{M}).$$
    This justifies defining
    \begin{align*}
        \|f\|_{\Lp(\mathcal{M};\Hc)} = \|f\|_{\Lp(\mathcal{M}\overline{\otimes} B(H))}=\|(f^*f)^{1/2}\|_{\Lp(\mathcal{M})}
    \end{align*}
    on $\Lp(\mathcal{M};\Hc)$. Analogously, the \emph{row Hilbert-valued $\Lp$-space} is
    \begin{align*}
        \Lp(\mathcal{M};H^{*r}) = (\mathbf{1}_{\mathcal{M}} \otimes p_{\mathds{1}}) \Lp(\mathcal{M} \overline{\otimes} B(H))
    \end{align*}
    equipped with the norm 
    \begin{align*}
        \|f\|_{\Lp(\mathcal{M};H^{*r})} = \|(f f^*)^{1/2}\|_{\Lp(\mathcal{M})}=\|f^*\|_{\Lp(\mathcal{M};\Hc)}.
    \end{align*}
       
    Given elements $h\in H$ and $m\in \Lp(\mathcal M)$, we will write
    $m\otimes h \in \Lp(\mathcal{M}) \otimes H$ for the element
    $m \otimes (h\otimes \overline{\mathds{1}})\in \Lp(\mathcal{M};\Hc)$ and similarly
   $m\otimes \overline h$  means $m \otimes (\mathds{1}\otimes \overline h)\in \Lp(\mathcal{M};H^{*r})$.With these notations,
    for $f=\sum_{i=1}^n m_i\otimes h_i$ and $g=\sum_{i=1}^n m_i\otimes \overline h_i$:
    \begin{eqnarray*}
      \|f\|_{\Lp(\mathcal{M};\Hc)} &=&  \Big\| \Big( \sum_{i,j=1}^n \langle h_i,h_j\rangle_H \ m_i^* m_j \Big)^{1/2}\Big\|_{\Lp(\mathcal{M})},\\
      \|g\|_{\Lp(\mathcal{M};H^{*r})} &=& \Big\| \Big( \sum_{i,j=1}^n \langle h_i,h_j\rangle_H \ m_i m_j^* \Big)^{1/2}\Big\|_{\Lp(\mathcal{M})}.
    \end{eqnarray*}
    In this way, we do not need to refer to the function $\mathds{1}$. Moreover, we have that $(m\otimes h)^*= m^* \otimes \overline h$ as operators. We will use without reference that $\Lp(\mathcal M)\otimes H$ is dense (resp. weak$^*$ dense) in $\Lp(\mathcal{M};\Hc)$ for $1\leq p<\infty$ (resp. $p=\infty$) and similarly for rows.

    In fact, column and row Hilbert-valued $\Lp$-spaces satisfy the
    expected duality relations expressed via the natural duality bracket 
    \begin{align}
      \la f, g\ra_{r,c} = \mathrm{Tr\otimes \tau }(fg),& &
     \la m\otimes \overline h, m'\otimes h'\ra_{r,c}= \tau (mm') \la h,h'\ra . \label{eq:dualityBracket}
    \end{align}
    In particular, it holds linearly isometrically
    \begin{align*}
        \Lp(\mathcal{M};\Hc)^* = \Lpprime(\mathcal{M};H^{*r}) \ \mathrm{and} \ \Lp(\mathcal{M};H^{*r})^* = \Lpprime(\mathcal{M};\Hc).
    \end{align*}
for any $1 \leq p < \infty$ whenever $1/p + 1/p^\prime =1$. On the other hand, we recall the following fact about homogeneity of column and row Hilbert spaces, see \cite[Lemma 2.4]{JLX06}. We state it only for columns.

    \begin{corollary}\label{cor:extension}
      Let $H$ and $K$ be two Hilbert spaces, and let
      $T : H \longrightarrow K$ be a bounded linear operator. Then
      $\id_{L_p(\mathcal{M})} \otimes T$ admits a unique continuous 
      extension (resp. weak$^*$ continuous) $\widetilde{T}$ from $\Lp(\mathcal{M};\Hc)$ into
      $\Lp(\mathcal{M};\Kc)$ for $1\leq p<\infty$ (resp. $p=\infty)$ with the same norm. \end{corollary}

    Some obvious properties of these extension maps will be crucial in the following sections.

    \begin{lemma}\label{lem:extensionProperties}
        Let $H$ and $K$ be two Hilbert spaces, let $S$, $T$, and $(T_j)_{j=1}^{\infty}$ be some bounded operators from $H$ to $K$ and let $\widetilde{S}, \widetilde{T}, (\widetilde{T_j})_{j=1}^{\infty}$ be the corresponding extensions from $\Lp(\mathcal{M};H^c)$ to $\Lp(\mathcal{M};K^c)$. Then the following holds.
        \begin{enumerate}
            \item $\widetilde{ST} = \widetilde{S}\widetilde{T}$,
            \item If $S$ and $T$ commute, then $\widetilde{S}$ and $\widetilde{T}$ also commute,
            \item whenever $\sum_{j=1}^{\infty} T_j$ converges in the norm of $B(H)$, there holds
            \begin{align*}
                \widetilde{\sum_{j=1}^\infty T_j} = \sum_{j=1}^{\infty} \widetilde{T_j}.
            \end{align*}
            \item $({\widetilde{S}})^*= \widetilde{\overline S^*}$ for $p<\infty$.
        \end{enumerate}
    \end{lemma}

 In the last statement, one needs a conjugation because the $^*$ on the left is for the Banach space adjoint whereas the right one is the Hilbert space sense. Of course similarly statements hold for row spaces considering bounded maps $T: H^*\to  K^*$.
    \subsection{Noncommutative spaces $\Lp(\mathcal{M};L^c_2(\Omega))$}

    Let $(\Omega,\mu)$ be a $\sigma$-finite measure space. A remarkable setting for noncommutative Hilbert-valued column/row $\Lp$-spaces is the case $H = \Ltwo(\Omega):= \Ltwo(\Omega,\mu)$. Notice that under these conditions, the duality bracket \eqref{eq:dualityBracket} is given by the expression
    \begin{align*}
        \la m_1 \otimes \overline f_1 ,m_2 \otimes f_2\ra_{r,c} &= \tau_{\mathcal{M}}(m_1 m_2) \ \int_{\Omega} \overline{f_1} f_2 \ d\mu.
    \end{align*}
    In particular, identifying $\Ltwo(\Omega,\mu)^*$ and $\Ltwo(\Omega,\mu)$ and using the bilinear pairing $(f,g)\mapsto \int_\Omega fg \ d\mu$, then
    \begin{align*}
        \Lpprime(\mathcal{M};\Ltwor(\Omega,\mu)) = \Lp(\mathcal{M};\Ltwoc(\Omega,\mu))^* \ \mathrm{for \ } 1 \leq p < \infty.
    \end{align*}
    This duality identity will allow us to drop the conjugation or the involution $*$ from now on. Moreover, for
    $F=\sum_{i=1}^n m_i\otimes f_i \in \Lp(\mathcal M)\otimes
    L_2(\Omega)$ with $p<\infty$:
    $$\|F\|^p_{ \Lp(\mathcal{M};\Ltwoc(\Omega,\mu))} = \tau \Big(\int_\Omega \Big|\sum_{i=1}^n f_i(t) m_i\Big|^2 d\mu \Big)^{p/2}$$
    Since $\sum_{i=1}^n f_i(t) m_i$ can be interpreted as $F(t)$, it is very tempting to consider elements $\Lp(\mathcal{M};\Ltwoc(\Omega))$ as functions. Indeed for $p\leq 2$ (see \cite[Proposition 2.5]{JLX06}):

    \begin{proposition}\label{functions}Let $1\leq p \leq 2$, the identity on $\Lp(\mathcal M)\otimes L_2(\Omega)$
      extends to an injective contraction $\Lp(\mathcal{M};L_2^c(\Omega))\to
      \Ltwo(\Omega;\Lp(\mathcal{M}))$.
      \end{proposition}
      \begin{proof}
        First note that by definition
        $$\Ltwo(\mathcal M)\otimes_2 L_2(\Omega)=\Ltwo(\mathcal{M};L_2^c(\Omega))=\Ltwo(\Omega;\Ltwo(\mathcal{M})).$$
        Let $p<2$ and set $q$ be so that
        $\frac 1 p=\frac 1 2 +\frac 1 q$.  Take
        $F=\sum_{i=1}^n m_i\otimes f_i \in \Lp(\mathcal M)\otimes
        L_2(\Omega)\subset \Lp(\mathcal{M};L_2^c(\Omega))$, it
        corresponds to an element of
        $\Lp(\mathcal M \overline\otimes B(\Ltwo(\Omega)))$. Its
        modulus $|F|^2= \int_\Omega |F(t)|^2d\mu$ falls into $L_{p/2}(\mathcal M)$, thus
        using its polar decomposition it can be written as $F=ab$ with
        $b\in \Lq(\mathcal M)$ and
        $a\in \Ltwo(\mathcal{M};L_2^c(\Omega))$ with
        $\|a\|_2 \cdot \|b\|_q=\|F\|_p$. Actually $a$ is also a simple function  and
$$\int_\Omega \|F(t)\|_p^2 d\mu \leq \|b\|_q^2 \int_\Omega \|a(t)\|_2^2 d\mu=\|F\|_p^2.$$
This shows that the identity is indeed a contraction and thus extends to a contraction $\iota$ by density.

To show the injectivity, consider $m\in \Lq(\mathcal M)$ and
$h\in L_2(\Omega)$.  The linear form associated to
$m\otimes  h$ satisfies
$\la F, m\otimes  h\ra_{c,r}=\int_\Omega \tau(\iota(F)(t)m)
 h(t) d\mu$. Indeed, this is clear for
$F\in \Lp(\mathcal M)\otimes L_2(\Omega)$, and for all $F$ by density.
Since $\Lq(\mathcal M)\otimes L_2(\Omega)$ is norm-dense in $\Lq(\mathcal M;  L_2^r(\Omega))$, we can conclude. 
           \end{proof}

As a consequence, we can identify elements in $\Lp(\mathcal M;  \Ltwoc(\Omega))$ with a.e. Bochner measurable functions from $\Omega$ to $\Lp(\mathcal M)$ when $1\leq p\leq 2$. This will be convenient for some identifications. Unfortunately, when $p>2$ we have no way to consider elements in $\Lp(\mathcal{M};\Ltwoc(\Omega))$ as  functions. Indeed, in the following sections, the case $p=\infty$ will be specially relevant. For that reason, the  extension of some useful operators on $\Ltwo(\Omega)$ to $\Lp(\mathcal{M};\Ltwoc(\Omega))$ will be carefully studied.
    

    \begin{lemma}\label{lem:extensionMapsL2}
        Let $A,B$ be two measurable sets, and let $w$ and $w^\prime$ be strictly positive functions belonging to $\Linfty(\Omega,\mu)$. Consider the following operators on $\Ltwo(\Omega,\mu)$:
        \begin{align*}
            T_{w} &: f \mapsto w^{1/2} f, \\
            P_{A}&: f \mapsto \bigchi_{A} f.
        \end{align*}
        These maps extend to bounded operators on $\Lp(\mathcal{M};L^t_2(\Omega,\mu))$, for $t=c,r$ and any $1\leq p\leq \infty$ such that $\|\widetilde{T}_w\| = \|w\|_{\Linfty(\Omega,\mu)}$, $\|\widetilde{P}_A\| = \|\bigchi_{A}\|_{\Linfty(\Omega,\mu)}$. Moreover, they satisfy the following relations:
        \begin{enumerate}
            \item $\widetilde{T}_{w} \widetilde{T}_{w^\prime} = \widetilde{T}_{w^\prime} \widetilde{T}_{w}$,
            \item $\widetilde{P}_A \widetilde{T}_w = \widetilde{T}_w \widetilde{P}_A$,
            \item whenever $w^{-1}$ is bounded, $\widetilde{T}_w \widetilde{T}_{w^{-1}} = \id = \widetilde{T}_{w^{-1}} \widetilde{T}_w$.
            \item $\widetilde{P}_A = \widetilde{P}_A \widetilde{P}_B = \widetilde{P}_B \widetilde{P}_A$ whenever $A \subseteq B$,
            \item $\widetilde{P}_{B \setminus A} = \widetilde{P}_B - \widetilde{P}_A$ whenever $A \subseteq B$.
        \end{enumerate}
    \end{lemma}
    \begin{proof}
        By Corollary~\ref{cor:extension}, the extension operators $\widetilde{T}_w$, $\widetilde{P}_A$ are bounded as long as the original ones are bounded on $L_2(\Omega,\mu)$. The maps $T_w$ and $P_A$ are bounded with norm $\|w^{1/2}\|_{\Linfty}$ and $1$ respectively, since they are pointwise multiplication operators. Claims (1)-(4) follow from Lemma~\ref{lem:extensionProperties}, while linearity of the map $T \mapsto \widetilde{T}$ implies (5).
    \end{proof}

    \section{Duality between Hardy spaces and  $\BMO$-spaces}

    Consider the measure space $(\mathbb{R}, \frac{dt}{1+t^2})$. Set $\omega(t) = 1+t^2$. Then, $\Ltwo(\mathbb{R},\frac{dt}{1+t^2})$ is a Hilbert space with the inner product
   $$        \langle g, f \rangle_{1/\omega} := \int_{\mathbb{R}} \overline{g(t)} \ f(t) \ \frac{dt}{1+t^2}.$$
    We will consider the associated column space $\Lp(\mathcal{M};\Ltwoc(\mathbb{R},\frac{dt}{1+t^2}))$ for $0 < p \leq \infty$. We will choose $\mathds{1}$ to be the constant function $1/\sqrt{\pi}$, which satisfies the condition $\| \mathds{1}\|_{\Ltwo(\mathbb{R},\frac{dt}{1+t^2})} = 1$. For the sake of exposition, we define some operators on $\Linfty(\mathcal{M};\Ltwoc(\mathbb{R},\frac{dt}{1+t^2}))$ which can be described in terms of the maps appearing in Lemma~\ref{lem:extensionMapsL2}. Let $A$ be a measurable set with nonzero measure and define the map
   $$        R_A = \frac 1{|A|}\Tw  \widetilde{P}_A$$
    so that $R_A$ is the extension of the operator acting on $\Ltwo(\mathbb{R},\frac{dt}{1+t^2})$ as follows:
  $$
        f \mapsto \frac{(1+t^2)^{1/2}}{\sqrt{|A|}} \bigchi_A f.$$
    On the other hand, denote by $a_A$ the extension to $\Linfty(\mathcal{M};\Ltwoc(\mathbb{R},\frac{dt}{1+t^2}))$ of the map
   $$        a_A : f \mapsto f_A \ 1 = \Big(\frac{1}{|A|} \int_A f \Big) \ 1.$$
    Similarly the linear form $f\mapsto f_A$, has a weak$^*$ extension
    from $\Linfty(\mathcal{M};\Ltwoc(\mathbb{R},\frac{dt}{1+t^2}))$ to
    $\mathcal M= L_\infty(\mathcal M; \mathbb C^c)$. Thus we may also use the notation $f_A$ for every $f\in \Linfty(\mathcal{M};\Ltwoc(\mathbb{R},\frac{dt}{1+t^2}))$. 

    Some other Hilbert spaces over the real line will be considered through this work. Since the distinguished vectors are not relevant, we will always write them as $\mathds{1}$ in all spaces.

    \begin{lemma}\label{lem:pre-adjoint}
        Let $A$ be a measurable set. Then, the pre-adjoint maps for $R_A$ and $a_A$ on $\Lone(\mathcal{M};\Ltwor(\mathbb{R},dt/(1+t^2)))$ act as follows on  any operator $m \otimes  f \in \Lone(\mathcal{M}) \otimes \Ltwo(\mathbb{R},dt/(1+t^2))$
        \begin{align*}
            (R_A)_* &: m \otimes f \longmapsto m \otimes \frac{\sqrt{1+t^2}}{\sqrt{|A|}} \bigchi_{I}f \\
            (a_A)_* &: m \otimes f \longmapsto m \otimes  \langle 1, f \rangle_{1/\omega} \frac{1+t^2}{|A|}\bigchi_A.
        \end{align*}
        Moreover, the operator
        $$
            \begin{array}{cccc}
                V &: \Linfty(\mathcal{M};\Ltwor(\mathbb{R},(1+t^2)dt)) & \longmapsto & \Linfty(\mathcal{M};\Ltwor(\mathbb{R},\frac{dt}{1+t^2})) \\
                & m \otimes  f &\longmapsto & m \otimes  (1+t^2)f,
            \end{array}
        $$
        is an isometry and admits a pre-adjoint
        $$
            \begin{array}{cccc}
                V_* &: \Lone(\mathcal{M};\Ltwoc(\mathbb{R},\frac{dt}{1+t^2})) & \longmapsto & \Lone(\mathcal{M};\Ltwoc(\mathbb{R},(1+t^2)dt)) \\
                & m \otimes f  &\longmapsto & m \otimes \frac{f}{1+t^2}.
            \end{array}
        $$
      \end{lemma}
      \begin{proof}
    These are consequences of the same results on Hilbert spaces using Proposition \ref{lem:extensionProperties}. One just need to take care of conjugations which actually play no role here. 
    \end{proof}

    Now, we are ready to define the column and row $\BMO$-spaces.

    \begin{definition}
        Given a von Neumann algebra $\mathcal{M}$ with n.s.f. trace $\tau$, set the \emph{column $\BMO$-space}, denoted as $\BMOc(\mathbb{R},\mathcal{M})$, to be the subspace of operators $f$ in $\Linfty(\mathcal{M};L_2^c(\R,\frac{dt}{1+t^2}))$ satisfying
        \begin{align}\label{eq:BMOcnorm}
            \|f\|_{\BMOc} := \sup_{I \subseteq \mathbb{R}} \big\|R_I (\id- a_I) f \big\|_{\Linfty(\mathcal{M};L_2^c(\R;\frac{dt}{1+t^2}))} < \infty,
        \end{align}
        where the supremum is considered over finite intervals $I$ of $\mathbb{R}$. Likewise, the row $\BMO$-space, $\BMOr(\mathbb{R},\mathcal{M})$, is the subspace of elements in $\Linfty(\mathcal{M};\Ltwor(\mathbb{R},\frac{dt}{1+t^2}))$ for which the norm $\|f\|_{\BMOr} := \|f^*\|_{\BMOc}$ is finite.
    \end{definition}

    It is clear that $\|\cdot\|_{\BMOc}$ is a norm modulo
    $\mathcal{M}$. This expression will be convenient for abstract
    questions. We point out that it admits a much more tractable form.
    The multiplication by $\sqrt{1+t^2}$ is just an isometry from
    $\Ltwo(\R)$ to $\Ltwo(\R,\frac{dt}{1+t^2})$. For a finite interval $I$,
    let $\widetilde{\iota}_I$ denote the extension of the map
    $f\mapsto \bigchi_I f $ from $\Ltwo(\R, \frac{dt}{1+t^2})$ to $\Ltwo(\R)$, then clearly:
 \begin{align}\Big\|R_I (\id- a_I) f \Big\|_{\Linfty(\mathcal{M};L_2^c(\R,\frac{dt}{1+t^2}))} =\Big\|\frac 1 {\sqrt{|I|}}\big( \widetilde \iota_I\big (f -(f_I\otimes  1)\big)\big)\Big\|_{\Linfty(\mathcal{M};L_2^c(\R))}\label{eq:simple}\end{align}
In particular, for operators in $\mathcal{M} \otimes \Ltwo(\mathbb{R},\frac{dt}{1+t^2})$, we recover the expression which determines the definition for the $\BMOc$-norm introduced in \cite{M07}:
    \begin{lemma}
        For any operator $f$ in $\mathcal{M} \otimes \Ltwo(\mathbb{R},\frac{dt}{1+t^2})$ it holds
    $$
            \|f\|_{\BMOc} = \sup_{I \subseteq \mathbb{R}} \Bigg\| \Bigg( \frac{1}{|I|} \int_{I} |f - f_I|^2 \Bigg)^{\frac12} \Bigg\|_{\mathcal{M}}.
        $$\end{lemma}

    Along the next section, the study of the boundedness of Calder\'on-Zygmund operators on operator-valued Hardy spaces will require a concrete formulation in terms of atomic decompositions. In order to justify introducing these spaces, we will check that the dual of this new description of the column (resp. row) Hardy space coincides with $\BMOr(\mathbb{R},\mathcal{M})$ (resp. $\BMOc(\mathbb{R},\mathcal{M})$).

    \begin{definition}\label{def:catom}
        Let $\mathcal{M}$ be a von Neumann algebra with n.s.f. trace. A \emph{$c$-atom} is a function $a \in \Lone(\Avnat)$ which admits a factorization of the form $a = b h$ for some function $b : \R \rightarrow \Ltwo(\mathcal{M})$ and an norm-one operator $h \in \Ltwo(\mathcal{M})$, satisfying
        \begin{enumerate}
            \item $\mathrm{supp}_{\mathbb{R}}(b) \subseteq I$ for some interval $I$,
            \item $\int_{I} b = 0$,
            \item $\|b\|_{\Ltwo(\R;\Ltwo(\mathcal{M}))} \leq \frac{1}{\sqrt{|I|}}$.
        \end{enumerate}
        Then, the \emph{column Hardy space} $\Hardyc$ is defined to be the subspace of elements in $\Lone(\Avnat)$ of the form
        \begin{align}
            \sum\limits_{i=0}^{\infty} \lambda_i a_i \ \mathrm{where \ } (\lambda_i)_i \in \ell_1 \ \mathrm{and} \ (a_i)_i \ c\mbox{-}\mathrm{atoms}
        \end{align}
        with respect to the norm
        \begin{align*}
            \|f\| = \inf \{ \sum_{i=0}^{\infty} |\lambda_i| \ : \ f = \sum_{i = 0}^{\infty} \lambda_i a_i \ \mathrm{in \ } \Lone(\Avnat), \ (\lambda_i)_i \in \ell_1, \ (a_i)_i \ c\mbox{-}\mathrm{atoms} \}.
         \end{align*}
    \end{definition}

    Under the above definition, any $c$-atom satisfies
    $$
    \|a\|_{\Lone(\R;\Lone(\mathcal{M}))} \leq 1
    $$
    since, by the Hölder inequality,
    \begin{align*}
        \|a\|_{\Lone(\Avnat)} \leq \|b\|_{\Ltwo(\R;\Ltwo(\mathcal{M}))} \ \|h  \bigchi_{I}\|_{\Ltwo(\R;\Ltwo(\mathcal{M}))} \leq |I|^{-1/2} \ |I|^{1/2} = 1.
    \end{align*}
    Therefore, $\Hardyc$ is contractively contained into $\Lone(\Avnat)$.
    
  We choose to give the explicit decomposition of atoms as $a=bh$ rather than just saying $a\in L_1(\mathcal M;\Ltwoc(\R))$ with norm less than $\frac 1 {\sqrt {|I|}}$ and mean zero. We do so as this makes explicit the connection with vector valued harmonic analysis and will make all the proofs transparent. We leave the following to the reader:

    \begin{proposition}\label{prop:H1cBanach}
        The column Hardy space $(\Hardyc,\|\cdot\|_{\Honec})$ is a Banach space.
    \end{proposition}

    Given a Banach space $\mathbb{X}$, let $\Ltwozero(\mathbb{R},(1+t^2) \ dt;\mathbb{X})$ denote the subspace of functions $f$ in $\Ltwo(\mathbb{R},(1+t^2) \ dt;\mathbb{X})$ satisfying
    \begin{align*}
        \int_{\mathbb{R}} f(t) \ dt = 0.
    \end{align*}
    Then, the classical argument by Meyer \cite[Chapter 5, Proposition 1]{Mey90} extends to the Banach-valued setting yielding the inclusion as a subspace of $\Ltwozero(\mathbb{R},(1+t^2) \ dt;\mathbb{X})$ into the vector-valued Hardy space $\Hone(\mathbb{R};\mathbb{X})$ \cite{H06}. More clearly, given $f \in \Ltwozero(\mathbb{R},(1+t^{2})dt;\Ltwo(\mathcal{M}))$, there exists a sequence of atoms $(b_i)_i \subseteq \Hone(\mathbb{R};\Ltwo(\mathcal{M}))$ and $(\lambda_i)_i \in \ell_1$ such that
    \begin{align}\label{eq:atommeyer}
        f = \sum_{i=0}^\infty \lambda_i b_i \ \mathrm{in} \ \Lone(\mathbb{R};\Ltwo(\mathcal{M})) \ \mathrm{and} \ \sum_{i=0}^\infty |\lambda_i| \lesssim \|f\|_{\Ltwozero(\mathbb{R},(1+t^2)dt;\Ltwo(\mathcal{M}))}.
    \end{align}
    Since any $c$-atom $a = bh$ is the product of an $\Ltwo$-atom $b$ in $\Hone(\mathbb{R};\Ltwo(\mathcal{M}))$ and an element $h$ in $\Ltwo(\mathcal{M})$, the argument by Meyer still works in the semicommutative case. 
    \begin{proposition}\label{prop:meyerMap}
      The formal identity map
      $$
      \Ltwozero(\mathbb{R},(1+t^2)dt)\otimes \Lone(\mathcal M) \to
      \Lone(\mathbb{R})\otimes\Lone(\mathcal M)$$ extends to an injective and 
      contractive map $Q: \Lone(\mathcal M; \Ltwoczero(\mathbb{R},(1+t^2)dt))\to \Hardyc$.    \end{proposition}
    \begin{proof}
      First note that simple tensors
      $f\otimes m \in \Ltwozero(\mathbb{R},(1+t^2)dt)\otimes
      \Lone(\mathcal M)$ are both in
      $\Hardyc$ and in $\Lone(\mathbb{R})\otimes\Lone(\mathcal M)$. Indeed by
      Meyer's result in the scalar case, $f=\sum_i \lambda_i a_i$ where
      $a_i$ are scalar valued atoms. Choose
      $\alpha,\beta\in \Ltwo(\mathcal M)$ with $m=\alpha\beta$ and
      $\|\alpha\|_2=1$. Then $a_i\otimes\alpha$ are atoms in
      $\Hone(\mathbb{R};\Ltwo(\mathcal{M}))$, so $a_i\otimes m$ is a
      multiple of a $c$-atom and $f \otimes m$ is in $\Hardyc$.

      The end of the argument is as in Proposition \ref{functions}. Any
      simple
      $x=\sum_{i=1}^n f_i\otimes m_i\in
      \Ltwozero(\mathbb{R},(1+t^2)dt)\otimes \Lone(\mathcal M)$ can be
      written as $x=F \beta$ , with $\beta\in \Ltwo(\mathcal M)$ and a
      simple tensor
      $F\in \Ltwo(\mathcal M;\Ltwoczero(\mathbb{R},(1+t^2)dt))=
      \Ltwozero(\mathbb{R},(1+t^{2})dt;\Ltwo(\mathcal{M}))$ such that
      $ \|F\|_2 \cdot \|\beta\|_2=\|x\|_{\Ltwo(\mathcal
        M;\Ltwoczero(\mathbb{R},(1+t^2)dt))}$. Using Meyers'
      decomposition \eqref{eq:atommeyer} for $F$ gives the norm estimate. One
      concludes to the boundedness of the extension by density. The injectivity follows by the weak$^*$ density of simple tensors in the dual spaces as in Proposition~\ref{functions}. 
    \end{proof}

    Consequently any linear form on $\Hardyc$ induces another one on the space $\Lone(\mathcal M;\Ltwoczero(\mathbb{R},(1+t^2)dt))$. Precomposing it with the map $V_*$ of Lemma \ref{lem:pre-adjoint} (which obviously
    has dense range), this allows us to represent $\Hardyc^*$ as a subspace of 
    $\Lone(\mathcal M;\Ltwoczero(\mathbb{R},1/(1+t^2)dt))^*\subset
    \Linfty(\mathcal M;\Ltwor(\mathbb{R},1/(1+t^2)dt))$.

    \begin{theorem}\label{thm:duality1}
        Given a semifinite von Neumann algebra $\mathcal{M}$, we have a contractive inclusion
       $$
            \Hardyc^* \subseteq \BMOr(\mathbb{R},\mathcal{M}).$$
      \end{theorem}
    \begin{proof}
Let $g\in \Hardyc^*$, we already explained that 
$$VQ^*(g)\in  \Linfty(\mathcal M;\Ltwor(\mathbb{R},1/(1+t^2)dt)).$$ We estimate its $\BMOr$ norm, using that all operators commute with the involution
               \begin{align*}
            \|VQ^* g \|_{\BMOr} &= \sup_{I \subseteq \mathbb{R}} \|R_I (\id-a_I)(VQ^* g)\|_{\Linfty(\mathcal{M};\Ltwor(\mathbb{R},\frac{dt}{1+t^2}))} \\
            &= \sup_{I} \sup_{f} |\langle f, R_I(\id-a_I)(VQ^* g) \rangle |
        \end{align*}
        where the supremum is taken over $f$ in the unit ball of $ \Lone(\mathcal{M};\Ltwoc(\mathbb{R},\frac{dt}{1+t^2}))$. 
        We can as well take it on simple tensors by density (we could
        as well consider it as a function with values in
        $\Lone(\mathcal M)$).  Moreover by the factorization argument
        of Proposition \ref{functions}, we can even write it as
        $f=\sum_{i=1}^n m_ih\otimes f_i$ with
        $f_i\in \Ltwo(\mathbb{R},\frac{dt}{1+t^2})$,
        $m_i,h\in \Ltwo(\mathcal M)$ such that
        $\|h\|_2= \|\sum_{i=1}^n m_i\otimes f_i\|_{\Ltwo(\mathbb{R},\frac{dt}{1+t^2};\Ltwo(\mathcal M))}=1$.
        Using the
        formulas of Lemma \ref{lem:pre-adjoint},
        $$QV_*(\id-a_I)_*R_{I*}(f)=\sum_{i=1}^n m_i h\otimes F_i$$ 
        where for $i=1,...,n$
        \begin{align*}
          F_i = \frac{\bigchi_{I}}{\sqrt{|I|}\sqrt{1+t^2}} f_i &- \frac{\bigchi_{I}}{|I|}  \langle 1, \frac{\sqrt{1+t^2}}{\sqrt{|I|}} \bigchi_{I}f_i \rangle_{1/\omega} \in \Ltwo(\mathbb R, (1+t^2)dt).
        \end{align*}
    Set $b=\sum_{i=1}^n m_i\otimes F_i\in \Ltwo(\mathcal M;\Ltwo(\mathbb R, (1+t^2)dt))$. Then, clearly $\mbox{supp}_{\mathbb{R}}(b) \subseteq I$, $\int_{I} b = 0$ since  $\int_{I} F_i = 0$. Let $G=\frac{\bigchi_{I}}{\sqrt{|I|}\sqrt{1+t^2}}\sum_{i=1}^n m_i\otimes f_i$. Using that $b$ is the projection of $G$ onto the orthogonal of constant functions on $I$, one  gets
        $$ \Big( \int_{I} \big\| b(t) \big\|_2^2 \ dt \Big)^{1/2} \leq
        \big\|G\big\|_{L_2(\mathbb R)}\leq \frac{1}{\sqrt{|I|}} \Big\|\sum_{i=1}^n m_i\otimes f_i\Big\|_{\Ltwo(\mathbb{R},\frac{dt}{1+t^2};\Ltwo(\mathcal M))}\leq \frac{1}{ \sqrt{|I|}}.
        $$
        In other words, $QV_*(\id-a_I)_*R_{I*}(f)$ is a $c$-atom, hence
        $$            \|VQ^*g\|_{\BMOr}\leq  \|g\|_{\Hardyc^*}.$$\end{proof}

        The reverse inclusion $\BMOr(\mathbb{R},\mathcal{M}) \subseteq \Hardyc^*$ is more involved. We need to check that every operator in $\BMOr(\mathbb{R},\mathcal{M})$ induces a continuous functional on $\Hardyc$. The starting point of our argument is that every operator $\varphi \in \BMOr(\mathbb{R},\mathcal{M})$ induces a functional on the algebraic vector space $\mathcal H$ generated by the $c$-atoms in $\Hardyc\subset \Lone(\Avnat)$.

    Recall that $\Lone(\mathcal{M};\Ltwoc(\mathbb{R},dt))$ can also
    be interpreted as a space of functions (defined a.e.) with values
    in $\Lone(\mathcal M)$. By definition, as a function any $c$-atom
    $a$ is the pointwise product of a compactly supported function in
    $\Ltwo(\mathbb R;\Ltwo(\mathcal{M}))$ with a constant element
    $h\in \Ltwo(\mathcal{M})$, thus it is an element in
    $\Lone(\mathcal{M};\Ltwoc(\mathbb{R},dt))$ and moreover
    $\|a\|_{\Lone(\mathcal{M};\Ltwoc(\mathbb{R},dt))}\leq 1$. To
    emphasize, we will denote the inclusion by
    $\gamma: \mathcal H\to
    \Lone(\mathcal{M};\Ltwoc(\mathbb{R},dt))$. The arguments we just
    gave also say that it is continuous for the norm on $\mathcal H$
    given by
    $$\|f\|_{\mathcal H}= \inf\{ \sum_{i=1}^N |\lambda_i|  \; : \; f=\sum_{i=1}^N \lambda_i a_i \; , N\geq 1, \; \lambda_i\in \C \textrm{ and } a_i \; c\textrm{-atoms} \}.$$
 
    We will denote by $A_g$ the pointwise multiplication by $g$ on
    a.e. functions with values in $\Lone(\mathcal{M})$ or
    $\Ltwo(\mathcal{M})$. It is well defined on all elements that
    have compact support in the spaces we consider. The same kind of arguments justify that
    $A_\omega(a)\in \Lone(\mathcal{M};\Ltwoc(\mathbb{R},1/(1+t^2)dt))$
    and
    $A_\omega : \mathcal H\to
    \Lone(\mathcal{M};\Ltwoc(\mathbb{R},1/(1+t^2)dt))$ is of course
    linear. $A_\omega$ is not continuous but if $a$ is an atom supported on $I$
    then $A_\omega(a)=\widetilde M_\omega^I (\gamma(a))$, where $M_\omega^I: \Ltwo(\R)\to L_2(\R, 1/(1+t^2)dt)$ is the multiplication by $\omega\bigchi_I$ (which is bounded).

    This allows us to define a duality pairing for $\mathcal H$ and
    $\BMOr$ by $\varphi \in \BMOr(\mathbb{R},\mathcal{M})$, $f\in \mathcal H$:
$$\la \varphi, f\ra_{\BMOr,\mathcal H} =\langle
\varphi, A_\omega(f)
\rangle_{\Linfty(\mathcal{M};\Ltwor(\mathbb{R},\frac{dt}{1+t^2})),
  \Lone(\mathcal{M};\Ltwoc(\mathbb{R},\frac{dt}{1+t^2})) }.$$
This definition is made so that if
$\varphi=\sum_{i=1}^n m_i\otimes f_i \in \mathcal M\otimes \Ltwo(\mathbb{R},\frac{dt}{1+t^2})$ then
$$\la \varphi, f\ra_{\BMOr,\mathcal H} =\int_{\mathbb R}  \tau (f(t)m_i) f_i(t) dt.$$ 
So, we recover the classical duality pairing for functions.
    
    \begin{lemma}\label{lem:functionalonAtoms}
        Let $\varphi \in \BMOr(\mathbb{R},\mathcal{M})$. If $a$ is a $c$-atom in $\Hardyc$, there holds
        \begin{align*}
            |\langle \varphi, a \rangle_{\BMOr,\mathcal H}| \leq \ \|\varphi\|_{\BMOr}.
        \end{align*}
    \end{lemma}
    \begin{proof}
      Let $I$ be the interval for which $a=bh$ satisfies the
      definition of a $c$-atom. Then $A_\omega(a)=|I|(R_{I*})^{2}(a)$ and
      the mean zero condition exactly means that $a_{I*}(A_\omega(a))=0$. Thus, there holds
      $$\langle \varphi, a \rangle_{\BMOr,\mathcal H}=\la R_I(\id-a_I)\varphi, |I| R_{I*}(a)\ra_{\Linfty(\mathcal{M};\Ltwor(\mathbb{R},\frac{dt}{1+t^2})),
        \Lone(\mathcal{M};\Ltwoc(\mathbb{R},\frac{dt}{1+t^2}))}.$$
      It remains to show that $\||I|R_{I*}(a)\|_{\Lone(\mathcal{M};\Ltwoc(\mathbb{R},\frac{dt}{1+t^2}))}\leq 1$, but this comes from the factorization property in terms of operators
      $$|I|R_{I*}(a)= (\sqrt{|I|} A_{\sqrt\omega}(b)) \cdot h$$ 
      where $\|\sqrt{|I|} A_{\sqrt\omega}(b)\|_{\Ltwo(\mathcal{M};\Ltwoc(\mathbb{R},\frac{dt}{1+t^2}))}
      =\sqrt{|I|} \ \| b\|_{\Ltwo(\mathcal{M};\Ltwoc(\mathbb{R},{dt}))}\leq 1$
      and $h$ is a norm-one operator.           
    \end{proof}

The second step in the proof is to extend the duality pairing to $\Hardyc$. This will require some care and approximations.

Given  a compactly supported continuous function $\xi:\mathbb R\to \R$, 
we denote by $R_\xi$ the convolution on
$\Ltwo(\mathbb R)$ against $\xi$. We could consider its
extension $\widetilde R_\xi$ to any $\Lp(\mathcal M;\Ltwoc(\mathbb
R))$. On the other hand, using vector-valued integration, as the convolution against $\xi$ in $\Lp(\mathbb \R; \Lp(\mathcal M))$ is feasible; we denote it by  $\xi*x$. When $p=2$, we have $\widetilde R_\xi(x)=\xi*x$ using the identification  $\Ltwo(\mathbb \R; \Ltwo(\mathcal M))=\Ltwo(\mathcal M;\Ltwoc(\mathbb R))$. Moreover, we consider here a third possible extension which will be used along this section.

\begin{lemma}\label{lem:conv}
  Assume $\xi$ is a compactly supported continuous function on $\R$. Given any $f\in \Lone(\mathbb \R; \Lone(\mathcal M))$, then $\xi*f\in\Lone(\mathcal M;\Ltwoc(\mathbb R))$ (viewed as a function
  space). Moreover, this induces a continuous linear map
  $C_\xi : \Lone(\mathbb \R; \Lone(\mathcal M))\to\Lone(\mathcal
  M;\Ltwoc(\mathbb R))$.
      \end{lemma}
      \begin{proof}
        Recall that $\Lone(\mathbb R; \Lone(\mathcal M))$ coincides with the projective tensor product $\Lone(\R) \widehat \otimes_\pi \Lone(\mathcal M)$. It suffices to prove the statement for a simple tensor
        $f=g \otimes m$ with $g\in\Lone(\R)$ and
        $m\in \Lone (\mathcal M)$. Then the function $\xi*f$
        corresponds to the operator
        $m\otimes (\xi*g)\in \Lone(\mathcal M;\Ltwoc(\mathbb
        R))$. Factorizing $m=rs$ with $r,s\in \Ltwo(\mathcal M)$ and
        $\|r\|_2=\|s\|_2=\|m\|_1^{1/2}$, we obtain a factorization in
        term of operators $m\otimes (\xi*g)= (r \otimes (\xi*g))s$
        so that $\|r \otimes (\xi*g)\|_{ \Ltwo(\mathcal M;\Ltwoc(\mathbb
        R))}=\|r\|_2 \|\xi*g\|_2$. Thus, we can conclude due to the Young inequality that
        $$\|m\otimes (\xi*g)\|_{ \Lone(\mathcal M;\Ltwoc(\mathbb
        R))}\leq \|r\|_2\|s\|_2\|\xi*g\|_2\leq \|m\|_1 \|g\|_1 \|\xi\|_2.$$
      We get that the norm of $C_\xi$ is controlled by $\|\xi\|_2$.
      \end{proof}
    Let fix some non-negative even continuous function $\phi:\mathbb R\to \R$ with support in $(-1,1)$ such that $\int_{\mathbb \R}\phi(t)dt =1$ and denote $\phi_n(x)=n\phi(nx)$.       
\begin{lemma}\label{lem:convolutionAtoms}
        Given a $c$-atom $a=bh$ supported on $I$, the following holds:
        \begin{enumerate}
        \item   $\frac 12\phi_n*a$ is a finite convex combination of $c$-atoms
        \item If $n\geq \frac 2{|I|}$, $\phi_n* a - a= \lambda_n a_n$,
          where $a_n$ is a $c$-atom supported on $2I$,  $\lambda_n\in \C$ with
          $\lim_{n\to \infty} \lambda_n=0$.
        \item We have $\gamma(\phi_n*a)=C_{\phi_n}(a)$ for all $n\geq 1$.
                  \end{enumerate}
      \end{lemma}

    \begin{proof}
     
      For (1), cut $(-\frac 1 n,\frac 1n)$ into
      $N$ disjoint intervals $K_j^n$ whose lengths are less than $|I|$ (we
      can have $N=1$) and consider $b_j^n=(\phi_n\bigchi_{K_j^n})*b$. For
      all $j$, $b_j^n$ is supported on an interval of length $2|I|$, has mean 0 and $L_2$-norm less than
      $\frac {2\|\phi_n\bigchi_{K_j^n}\|_1} {\sqrt{|I|}}$. Thus $b_j^n h$ is
      an atom up to a factor $2\|\phi_n\bigchi_{K_j^n}\|_1$ and
      $\phi_n * a=\sum_{j=1}^N b_j^n h$. But there holds
      $\sum_{j=1}^N 2\|\phi_n\bigchi_{K_j^n}\|_1=2$, yielding (1).

      Similarly, for $(2)$, $\phi_n*b-b$ has support in $2I$ with mean
      $0$. Its $\Ltwo$-norm goes to $0$ giving the result. 

      There are many ways to check (3). One can identify them as
      a.e. functions from $\R$ to $\Lone(\mathcal M)$. Another one is to
      note that point (1) implies that the convolution by $\phi_n$ is
      continuous for $\|\cdot\|_{\mathcal H}$. Moreover, atoms of the form
      $a=bh$, with $b$ a $\Ltwo$-atom which is a simple tensor in
      $\Ltwo(\R)\otimes \Ltwo(\mathcal M)$, and
      $h\in \Ltwo(\mathcal M)$ are dense in $\mathcal H$ for
      $\|\cdot\|_{\mathcal H}$.  On such atoms, the formula is obvious
      and then extends by continuity.
    \end{proof}

       \begin{proposition}\label{prop:BMOCompact}
         Let $\psi \in \BMOr(\mathbb{R},\mathcal{M})$ satisfying
         $\widetilde{P}_J \psi = \psi$ for some finite interval
         $J$. Then for any family of atoms $(a_i)$ and
         $(\lambda_i)\in \ell_1$ such that $\sum_{i=1}^\infty  \lambda_i a_i=0$ in  $\Lone(\Linfty(\mathbb{R}) \overline{\otimes} \mathcal{M})$ there holds
     $$            \sum_{i=1}^\infty  \lambda_i \ \langle \psi, a_i \rangle_{\BMOr,\mathcal H} = 0.$$
           \end{proposition}
    \begin{proof}
      First, by Lemma \ref{lem:convolutionAtoms}, $\phi_n* a_i\in \mathcal H$ for all $n, i\geq 1$. We come back to the definition
              \begin{eqnarray*}
          \langle \psi, \phi_n*a_i \rangle_{\BMOr,\mathcal H} &=&  \langle
\widetilde P_J \psi, \widetilde P_JA_\omega(\phi_n*a_i)
\rangle_{\Linfty(\mathcal{M};\Ltwor(\mathbb{R},\frac{dt}{1+t^2})),
  \Lone(\mathcal{M};\Ltwoc(\mathbb{R},\frac{dt}{1+t^2})) }.
         \end{eqnarray*}
Using the projection $\widetilde{P}_J$, we get 
         $\widetilde P_JA_\omega(\phi_n*a_i)=\widetilde P_J \widetilde
         M_\omega^J \gamma(\phi_n*a_i)$ and since $\gamma(\phi_n*a_i)=C_{\phi_n}(a_i)$:
 \begin{eqnarray*}
          \langle \psi, \phi_n*a_i \rangle_{\BMOr,\mathcal H} &=&  \langle
(\widetilde M_\omega^J)^* \psi, C_{\phi_n}(a_i)
\rangle_{\Linfty(\mathcal{M};\Ltwor(\mathbb{R})),
                                                                  \Lone(\mathcal{M};\Ltwoc(\mathbb{R})) }\\ &=&  \langle 
 C_{\phi_n}^*(\widetilde M_\omega^J)^* \psi, a_i
\rangle_{\mathcal{M}\overline \otimes \Linfty(\mathbb{R}),  \Lone(\mathcal{M}\overline \otimes \Linfty(\mathbb{R})) }.   
         \end{eqnarray*}
         We can conclude as $\sum_{i=1}^\infty \lambda_i a_i=0$ in $\Lone(\R; \Lone(\mathcal M))$ that $\sum_{i=1}^\infty \lambda_i\la \psi, \phi_n*a \ra=0$.

         Next, by Lemma \ref{lem:functionalonAtoms}, the duality
         pairing is continuous for the norm $\|\cdot\|_{\mathcal H}$. Lemma
       \ref{lem:convolutionAtoms} (1) and (2) mean that
       $\|\phi_n*a_i\|_{\mathcal H}\leq 2 \|a_i\|_{\mathcal H}$ and
   $\|\phi_n*a_i-a_i\|_{\mathcal H}\to 0$ as $n\to \infty$. We can conclude to     $\sum_{i=1}^\infty \lambda_i\la \psi, a_i \ra=0$ by the Lebesgue Theorem.
     \end{proof}

     In order to conclude the argument, it remains to show that any
     operator $\varphi \in \BMOr(\mathbb{R},\mathcal{M})$ induces a
     well-defined functional on $\Hardyc$. The estimates obtained for
     compactly supported $\psi \in \BMOr(\mathbb{R},\mathcal{M})$ will
     yield a general result under another suitable approximation
     argument from Garnett's book \cite[p. 261]{G07} as stated in \cite[Lemma 1.5]{M07}
     in the semi-commutative case. See Appendix~\ref{app:garnett} for its proof.

    \begin{lemma}\label{lem:garnett}
        Suppose $\varphi \in \BMOc(\mathbb{R},\mathcal{M})$ and suppose $J$ is an interval such that $\varphi_J = 0$. Let $3J$ be the interval concentric with $J$ having length $3|J|$. Then there exist some universal $C>0$, $\psi \in \BMOc(\mathbb{R},\mathcal{M})$ such that 
       $$\widetilde{P}_{3J} \psi = \psi,\qquad \widetilde{P}_J (\psi - \varphi) = 0, \qquad \|\psi\|_{\BMOc} \leq C \ \|\varphi\|_{\BMOc}.$$
                \end{lemma}    
    
    \begin{theorem}
        Given a semifinite von Neumann algebra $\mathcal{M}$, there is a contractive inclusion 
        $$            \BMOr(\mathbb{R},\mathcal{M}) \subseteq \Hardyc^*.$$
    \end{theorem}
    \begin{proof}
      
      Let $\varphi \in \BMOr(\mathbb{R},\mathcal{M})$ and let
      $f \in \Hardyc$ admit an atomic decomposition
      $\sum_{i=1}^\infty \lambda_i a_i$. We want to define the duality pairing by
      $$\langle \varphi, f\rangle = \sum_{i=1}^\infty \lambda_i \langle \varphi,a_i\rangle_{\BMOr,\mathcal H}.$$
      We start by showing that this is well defined algebraically, that
      is, if $f=0$ in
      $\Lone(\Linfty(\mathbb{R}) \overline{\otimes} \mathcal{M})$ then
      the right hand side is also 0.
      
 Given $\varepsilon > 0$, let $N \geq 1$ such that
      $\sum_{i > N} |\lambda_i| < \varepsilon$ and let $J$ be a finite
      interval satisfying for $i=1,...,N$:
        \begin{align*}
            \mathrm{supp}_\R(  a_i ) \subseteq J.
        \end{align*}
        Without loss of generality, we can assume that $\varphi_J = 0$, so by Lemma~\ref{lem:garnett}, there exists some $\psi \in \BMOr(\mathbb{R},\mathcal{M})$ satisfying $\widetilde{P}_{3J}\psi = \psi$, $\widetilde{P}_{J}(\varphi-\psi)=0$ and 
        $$
        \|\psi\|_{\BMOr} \leq C \|\varphi\|_{\BMOr}
        $$
        for some universal constant $C>0$. Therefore, as a consequence of Proposition~\ref{prop:BMOCompact},
        \begin{align*}
            \sum_{i=1}^\infty \lambda_i \langle \varphi, a_i \rangle &=  \sum_{i=1}^\infty \lambda_i \langle \varphi, a_i\rangle - \sum_{i=1}^\infty \lambda_i \langle \psi, a_i \rangle \\
            &= \sum_{i=1}^N \lambda_i \ \langle \varphi - \psi, a_i \rangle + \sum_{i=N+1}^\infty \lambda_i \ \langle \varphi - \psi, a_i \rangle \\
            &= \sum_{i=N+1}^\infty \lambda_i \,\langle \varphi - \psi, a_i \rangle.
        \end{align*}
        We get to
        \begin{align*}
            |\sum_{i=1}^\infty \lambda_i \ \langle \varphi, a_i \rangle | &\leq \sum_{i=N+1}^\infty |\lambda_i| (1 + C) \|\varphi\|_{\BMOr} < (1+C)\varepsilon \|\varphi\|_{\BMOr}.
        \end{align*}
        Since $\varepsilon$ is arbitrarily small, the duality pairing is well defined. To conclude the norm estimate, one just need to use Lemma \ref{lem:functionalonAtoms}.
    \end{proof}

    This last theorem completes the proof of the isometric identity
    \begin{align*}
        (\Hardyc)^* = \BMOr(\mathbb{R},\mathcal{M}).
    \end{align*}
    Therefore, a new description for a predual of
    $\BMOr(\mathbb{R},\mathcal{M})$ has been obtained just in terms of
    a new atomic decomposition, which will be crucial for the study of
    the boundedness of Calder\'on-Zygmund operators from $\Hardyc$ to
    $\Lone(\Avnat)$.

    \medskip
    
    We end this section by briefly mentioning that there is another approach
    to define the duality bracket $\Hardyc$-$\BMOr(\mathbb{R},\mathcal{M})$.
    One can easily show that there is a natural isometric embedding
    $$\pi : \Linfty(\mathcal M;\Ltwor(\R, 1/(1+t^2)dt))\subset B\big(\Ltwo(\R, \frac{dt}{1+t^2} \otimes_2\Ltwo(\mathcal M)), \Ltwo(\mathcal M)\big).$$
    Then for an atom $a=bh$ and $f\in \BMOr(\mathbb{R},\mathcal{M})$,
    one can set $\la a, f\ra=\tau( \pi(f)(b)h)$, without regard to the decomposition of $a$. Then, one can rely on Hilbert valued $\Hone$ and $\BMO$ noting that 
    $$\|f\|_{\BMOr(\mathbb{R},\mathcal{M})}=\sup_{\|h\|_2\leq 1} \| \pi(f)^*(h)\|_{\BMO(\R,\Ltwo(\mathcal M))}.$$
    Nevertheless, one still needs all the above approximation arguments.
    
       \section{Calder\'on-Zygmund operators with operator-valued kernels}

    Let $\mathcal{M}$ be a von Neumann algebra over a separable Hilbert space. Through this section we establish the conditions under which a kernel $K$, defined outside the diagonal with values in $\mathcal{M}$, will induce a Calder\'on-Zygmund operator from $\Hardyc$ into $\Lone(\Linfty(\mathbb{R}) \overline{\otimes} \mathcal{M})$. As one could expect, the resulting operator will be defined up to a pointwise multiplication operator. Recall the identification
    $\Linfty(I \times J) \overline{\otimes} \mathcal{M}=L_\infty(I \times J;\mathcal{M})$ where we consider weak$^*$-measurable $\mathcal M$-valued functions.

    Recall that we denote by $\mathcal S$ the set of compactly supported   essentially bounded functions $\mathbb R\to L_\infty\cap L_1(\mathcal M)$ that are measurable with values in $L_1$. Note that $B_{\Lp(\mathcal A)}\cap \mathcal S$ is norm-dense in $B_{\Lp(\mathcal A)}$ when $1\leq p<\infty$ and weak$^*$-dense for $p=\infty$. Indeed
      if $f\in \Lp(\mathcal A)$ with polar decomposition $f=u|f|$, take $p_n$ a sequence of finite projections strongly going to 1 in $\mathcal M$ and consider
      $f_n=u\max\{|f|,n\} \ p_n \ \bigchi_{[-n;n]}\in \mathcal S$. It is clear that $\|f_n\|_p\leq \|f\|_p$ and $f_n\to f$ in norm in $L_p$ (weak$^*$-convergence if $p=\infty$).

    \begin{definition}\label{def:CZ}
        Assume that $T$ is a bounded operator on $\Ltwo(\Linfty(\mathbb{R}) {\overline{\otimes}} \mathcal{M})$. We say that $T$ is a \emph{Calder\'on-Zygmund operator} if there exists some function
        \begin{align*}
            K : \mathbb{R} \times \mathbb{R} \ \setminus \ \{x=y\} \longrightarrow \mathcal{M}
        \end{align*}
        such that for every pair of intervals $I,J$ satisfying $d(I,J) > 0$, there exist $K_{I,J} \in \Linfty(I \times J) \overline{\otimes} \mathcal{M}$ such that
        \begin{itemize}
          \item ${K}_{I,J}(t) = K(t)$ for  almost every $t \in I \times J$,
 \item   for any $f,g \in \mathcal S$ supported respectively on $J$ and $I$:  \begin{align}\label{dissup}
            \Big(\tau \circ \int\Big)(g \ Tf) = \langle f(y)g(x), K_{I,J} \rangle_{\Lone(\Linfty(I\times J)\overline{\otimes} \mathcal{M}),\Linfty(I \times J)\overline{\otimes} \mathcal{M}}
        \end{align}
      \end{itemize}
      
        We say that $T$ is a \emph{left Calder\'on-Zygmund operator} if it also  satisfies that for any compactly supported $f \in \Ltwo(\Avnat)$ and $h \in \mathcal{M}$:
        \begin{align*}
            T(fh) = T(f)h. 
        \end{align*}

       Let $\lambda>0$, we say that the kernel $K$ (or $T$) satisfies the \emph{Hörmander condition} for $\lambda$ whenever
        \begin{align}\label{eq:hormanderCondition}
            \int_{|x-y| \geq \lambda |y^\prime-y|} \|K(x,y) - K(x,y^\prime)\|_{\mathcal{M}} \ dx \leq C_\lambda
        \end{align}
        for some constant $C_\lambda \geq 0$.  Regard that should the Hörmander condition hold for $\lambda > 0$, then it holds for $\lambda' > \lambda$.
    \end{definition}

    We choose $\mathcal S$ in order to have a perfectly defined duality  in \eqref{dissup}  as $f(y)g(x)\in L_1(\mathbb R^2;\Lone(\mathcal M))$. Moreover, the space $\mathcal S$ is an algebra for the pointwise product. We could also have chosen to state \eqref{dissup} for $L_1(\mathbb R; L_2(\mathcal M))$ instead of $\mathcal S$ without significant modifications.

    The following argument constitutes our main estimate and it can be considered as an arrangement of the scalar case to the semicommutative setting.

    \begin{lemma}\label{lem:CZatom}
        Let $\mathcal{M}$ be a von Neumann algebra. Let $T$ be a left Calder\'on-Zygmund operator which is bounded on $\Ltwo(\mathbb{R};\Ltwo(\mathcal{M}))$ and has associated kernel 
        $$
        K : \mathbb{R} \times \mathbb{R} \ \setminus \ \{x = y\} \longrightarrow \mathcal{M}.
        $$
        If $K$ satisfies the Hörmander condition for some $\lambda\geq 1$ and moreover
           \begin{align*}
            T(a) = T(b)h \ \textrm{\rm for  any $c$-atom}\  a=bh,
        \end{align*}
        then there holds
        \begin{align*}
            \|T(a)\|_{\Lone(\Linfty(\mathbb{R}) \overline{\otimes} \mathcal{M})} \leq \max\{C_\lambda, \lambda^{1/2} \|T\| \}.
        \end{align*}        
    \end{lemma}
    \begin{proof}
        Assume that $\lambda >1$. Let $a = bh$ be a $c$-atom in $\Hardyc$ with support contained in the interval $I=[y_0-d,y_0+d]$ and let $\lambda I = [y_0-\lambda d,y_0+\lambda d]$. Then, the norm
        \begin{align*}
            \|T(a)\|_{\Lone(\mathbb{R};\Lone(\mathcal{M}))} = \int_{\lambda I} \|T(a)\|_{\Lone(\mathcal{M})} + \int_{(\lambda I)^c} \|T(a)\|_{\Lone(\mathcal{M})}
        \end{align*}
        can be bounded in two steps. First, the continuity of $T$ on $\Ltwo(\mathbb{R};\Ltwo(\mathcal{M}))$ implies that
        \begin{align*}
            \int_{\lambda I} \|T(a)&\|_{\Lone(\mathcal{M})} = \int_{\lambda I} \|T(b)h\|_{\Lone(\mathcal{M})} \leq \Big( \int_{\lambda I} \|T(b)\|^2_{\Ltwo(\mathcal{M})} \Big)^{1/2} \ \Big( \int_{\lambda I} \|h\|_{\Ltwo(\mathcal{M})}^2 \Big)^{1/2} \\
            &= \|T(b)\|_{\Ltwo(\mathbb{R};\Ltwo(\mathcal{M}))} \ |\lambda I|^{1/2} \|h\|_{\Ltwo(\mathcal{M})} \\
            &\leq \|T : \Ltwo(\mathbb{R};\Ltwo(\mathcal{M})) \rightarrow \Ltwo(\mathbb{R};\Ltwo(\mathcal{M}))\| \ \|b\|_{\Ltwo(\mathbb{R};\Ltwo(\mathcal{M}))} \ |\lambda I|^{1/2} \ \|h\|_{\Ltwo(\mathcal{M})} \\
            &\leq \lambda^{1/2} \ \|T : \Ltwo(\mathbb{R};\Ltwo(\mathcal{M})) \rightarrow \Ltwo(\mathbb{R};\Ltwo(\mathcal{M}))\|.
        \end{align*}

        To deal with the second term, we have
        $$ \|T(a)\bigchi_{(\lambda I)^c}\|_1 = \|T(b)h \bigchi_{(\lambda I)^c}\|_1 
        = \sup_{\substack{g \in \mathcal S,  \|g\|_{\infty} \leq 1 \\ \bigchi_{(\lambda I)^c}g=g}} \big| \tau \int T(b) h g \big|.$$
        Assume for the moment that $b\in \mathcal S$ and $h\in L_\infty \cap L_1(\mathcal M)$. Let $y_0$ be the center of $I$. Since $I$ and $(\lambda I)^c$ are disjoint measurable sets and $b$ has integral zero, it follows using \eqref{dissup} and the H\"ormander condition that
        \begin{align*}
            \|T(a)&\bigchi_{(\lambda I)^c}\|_1 \leq \sup_g \big| \tau \int_{(\lambda I)^c} \int_{I} K_{(\lambda I)^c,I}(x,y) \ b(y) h g(x) \ dx dy \big| \\
            &\leq \sup_g \big| \tau \int_{(\lambda I)^c} \int_{I} (K_{(\lambda I)^c,I}(x,y) - K_{(\lambda I)^c,I}(x,y_0)) \ b(y) h g(x) \ dx \ dy \big| \\
            &= \big\| \int_{I} (K_{(\lambda I)^c,I}(x,y) - K_{(\lambda I)^c,I}(x,y_0)) \ b(y)h \ dy \big\|_{\Lone((\lambda I)^c;\Lone(\mathcal{M}))} \\
            &\leq \int_{(\lambda I)^c} \int_I \| (K_{(\lambda I)^c,I}(x,y) - K_{(\lambda I)^c,I}(x,y_0)) \ b(y)h \|_{\Lone(\mathcal{M})} \ dx \ dy \\
            &\leq \int_{(\lambda I)^c} \int_I \| (K_{(\lambda I)^c,I}(x,y) - K_{(\lambda I)^c,I}(x,y_0)) \|_{\mathcal{M}} \ \|b(y)\|_{\Ltwo(\mathcal{M})} \ \|h\|_{\Ltwo(\mathcal{M})} \ dx \ dy \\
            &\leq C_\lambda \ \int_I \|b(y)\|_{\Ltwo(\mathcal{M})} \ \|h\|_{\Ltwo(\mathcal{M})} \ dy 
            \leq C_\lambda \|b\|_{2} \ \|h \bigchi_I\|_{2}
            \leq C_\lambda.
        \end{align*}
        The general case follows by approximation. Indeed let
        $b_n\in \mathcal S$ going to $b$ with $\|b_n\|_2\leq \|b\|_2$ as
        explained above and $ h_n\in L_\infty\cap \Lone(\mathcal M)$
        going to $h$ in $L_2(\mathcal M)$ with
        $\|h_n\|_{\Ltwo(\mathcal M)}\leq 1$. Centering $b_n$ on $I$
        decreases its $L_2$-norm, so we can as well assume it has mean
        0. Let $J_k=[-k,k]\cap(\lambda I)^c$. Then for any $k\geq 1$,
                $T(b_n)h_n\bigchi_{J_k}\to T(b)h \bigchi_{J_k}$ in $\Lone$ as
                $n\to \infty$ by the $L_2$-continuity of $T$. It follows that
              $\| T(b)h \bigchi_{J_k}\|_1\leq \limsup_n  \| T(b_n)h_n \bigchi_{J_k}\|_1\leq C_\lambda$. Letting $k\to \infty$ gives $\|T(b)h \bigchi_{(\lambda I)^c}\|_1 \leq C_\lambda$.

        On the other hand, the case $\lambda =1$ can be recovered since the argument above works for $\lambda' > 1$, and taking the limit as $\lambda' \rightarrow 1$ yields the statement of the theorem.
      
      \end{proof}

    In order to extend a left Calder\'on-Zygmund operator to the whole Hardy space $\Hardyc$, we will proceed as in \cite{MeyC97}. A family of bounded kernels will be constructed, yielding a bounded family of Calder\'on-Zygmund operators that will approximate the original operator.
     
    \begin{lemma}\label{lem:approxKernel}
      Let $T$ be a left Calder\'on-Zygmund operator whose associated
      kernel $K$ satisfies the Hörmander condition  for some $\lambda \geq 1$\eqref{eq:hormanderCondition}. Then, there exist a sequence of left Calder\'on-Zygmund operators $(T_m)_{m\geq 1}$ with
      kernels
      $(K_m)_{m \geq 1} \subseteq \Linfty(\mathbb{R} \times
      \mathbb{R}) \overline{\otimes} \mathcal{M}$ and a constant
      $C'>0$ such that
        \begin{align}\label{eq:weakerHormander}
            \int_{|x-y| \geq 2\lambda |y^\prime-y|} \|K_m(x,y) - K_m(x,y^\prime)\|_{\mathcal{M}} \ dx \leq C'.
        \end{align}
         Moreover, there holds
            \begin{align*}
                \lim_{m \rightarrow \infty} \|T_m(f) - T(f) \|_{\Ltwo(\mathbb{R};\Ltwo(\mathcal{M}))} = 0
            \end{align*}
            for every $f \in \Ltwo(\Avnat)$.
    \end{lemma}

    \begin{proof} 
    We consider an even smooth function $\phi$ which is supported on $(-1,1)$ with
      $\int \phi = 1$ as in Lemma ~\ref{lem:convolutionAtoms}. Then, set $R_m$ to be the operator defined on
      $\Ltwo(\mathbb{R};\Ltwo(\mathcal{M}))$ as $R_m(f) = f * \phi_m$
      where $\phi_m(x) = m \phi(mx)$.
    
      We define the function $K_m:\R^2\to \mathcal M$ by duality. Assume
        $z\in \Lone(\mathcal M)$ is decomposed as $z=ab$ with $a,b\in \Ltwo(\mathcal M)$, and define the quantity 
      \begin{align}\label{formulaKm}
        \la K_m(x,y),z\rangle_{\mathcal M,\Lone(\mathcal M)}  =
        \int \tau\big(T((\tau_y\phi_m)(t)a)\tau_x ({\phi}_m)(t)b\big) dt,
      \end{align}
      where $\tau_x \phi(t) = \phi(t-x)$. Indeed assume $z=\alpha\beta$ with $\alpha,\beta\in \Ltwo(\mathcal M)$ and choose $p$ a
      finite projection so that $pa, p\alpha\in \mathcal M$, then using the right-modularity condition,
      $$\int_\R \tau\big(T(\tau_y\phi_mpa)\tau_x{\phi}_m \ b\big)= \int_\R \tau\big(T(\tau_y\phi_m p)\tau_x \phi_m \ z\big)=\int_\R \tau\big(T(\tau_y\phi_m p\alpha)\tau_x\phi_m \ \beta\big).$$
      Letting $p\to 1$ and using the $L_2$-continuity of $T$ yields
      $$\langle T(\tau_y\phi_ma), \tau_x ({\phi}_m)b \rangle_{\Ltwo(\R;\Ltwo(\mathcal M))}=\langle T(\tau_y\phi_m\alpha), \tau_x ({\phi}_m)\beta \rangle_{\Ltwo(\R;\Ltwo(\mathcal M))}.$$      
      Linearity is proved in a similar way. We also get
      \begin{align*}
            \|K_m(x,y)\|_{\mathcal{M}} &= \sup_{\substack{{a,b\in \Ltwo(\mathcal M)}\\ \|a\|_2=\|b\|_2=1}} | (\int \circ \tau) \big(T(\tau_y \phi_ma)  \ \tau_x {\phi}_m  b \big) | \\
            &\leq \|T\|_{\Ltwo(\R;\Ltwo(\mathcal M))}  \| \phi_m\|_{\Ltwo(\mathbb{R})}^2.
        \end{align*}
        Assuming that $f$ and $g$ are simple functions, we get by linearity 
\begin{align*}
  \langle R_m T R_m f , g \rangle &= \tau \int \int T\Big(\int \phi_m(\cdot-y) \ f(y) \ dy\Big)(t) \ \phi_m(t-x) \ dt \ g(x) \ dx\\
  &= \int\int \tau\big(K_m(x,y)f(y)g(x)\big) \ dx \ dy.
\end{align*}
The formula extends to $f,g\in \Ltwo(\R;\Ltwo(\mathcal M))$ by continuity since $K_m$ is bounded on the whole $\R^2$. Therefore, $K_m$ is the kernel for $R_m T R_m$ in the sense of Definition~\ref{def:CZ}. Moreover, $K_m$ is a continuous function $\R\times\R \to \Linfty(\mathcal M)$ by continuity of translations on $\Ltwo(\R)$ as
\begin{align}
  \nonumber \| K_m(x,y)-K_m(x,y')\|_\mathcal M&\leq                                           \|T\|  \ \| \phi_m\|_{2} \ \| (\tau_y-\tau_{y'})\phi_m\|_2 \\ \nonumber 
  &\leq \|T\| \ m^{1/2} \ \|\phi\|_2 \ \big( 2 m^{3/2}  {|y-y'|} \|\phi'\|_{\infty}\big)\\ \label{lipsest} & =
2  m^2 {|y-y'|} \|T\| \|\phi\|_2 \|\phi'\|_{\infty}
\end{align}
A similar Lipschitz estimate holds for $\| K_m(x,y)-K_m(x',y)\|_\mathcal M$. 
      
On the other hand, the kernel $K_m$ satisfies the Hörmander condition for $2\lambda$. First, consider $x,y,y'$ such that $|x-y| \geq 2\lambda |y^\prime - y|$
and $|x-y| > \frac{4}{m}$, then $|x-y'|\geq \frac{|x-y|}2>\frac 2 m$ as $\lambda\geq 1$. Thus the support of $\tau_x {\phi}_m$ is
disjoint from that of $\tau_y \phi_m$ and $\tau_{y'} \phi_m$. Let
$I_{y,y^\prime}$ and $I_x$ denote some disjoint closed intervals containing
the support of $\tau_y \phi_m,\tau_{y^\prime}\phi_m$ and
$\tau_x {\phi}_m$ respectively. Then, the identities \eqref{dissup} and \eqref{formulaKm} for the kernel
$K_{I_x,I_{y,y^\prime}}$ (extending to 0 outside $I_x\times I_{y,y^\prime}$) give
$$ K_m(x,z)= K_{I_x,I_{y,y^\prime}}*({\phi}_m\otimes \phi_m)(x,z)$$
for $z=y,y'$. Thus we can write
        \begin{align*}
            &\int_{|x-y| \geq 2\lambda |y^\prime - y|\vee \frac{4}{ m} } \|K_m(x,y)-K_m(x,y^\prime)\|_{\mathcal{M}} \ dx \\
            &\leq\int_{|x-y| \geq 2\lambda |y^\prime -y|\vee \frac{4}{m} } \int \int \big\| {K}(x - u,y -{v}) - {K}(x - {u},y^\prime - {v}) \big\|_{\mathcal{M}} \ {\phi_m}(u) \phi_m(v)\ du \ dv \ dx.
        \end{align*}
        Using that for $|u|,|v|\leq \frac{1}{m}$ (i.e. in the support of $\phi_m$), 
$$|x - u - (y -{v}) | \geq |x-y| - \frac{2}{m} > \frac{1}{2} |x-y| \geq \lambda |y^\prime-v -(y-v)|,$$
the H\"ormander condition for the kernel $K$ implies
    $$\int_{|x-y| \geq 2\lambda |y^\prime - y|\vee \frac{4}{ m}} \|K_m(x,y)-K_m(x,y^\prime)\|_{\mathcal{M}} \ dx \leq C.$$ 
    On the other hand, whenever $|x -y | \leq 4/ m$ holds, the estimate \eqref{lipsest} yields
        \begin{align*}
            &\int_{\frac{4}{m} \geq |x-y| \geq 2\lambda |y^\prime -y|} \|K_m(x,y)-K_m(x,y^\prime)\|_{\mathcal{M}} \ dx \\
            &\leq 2 m^2 \|T\| \|\phi^\prime\|_{\infty}\|\phi\|_2 \int_{\frac{4}{m} \geq |x-y| \geq 2\lambda |y^\prime -y|} |y^\prime -y| \ dx \leq \frac{8}{\lambda} \ \|T\| \|\phi^\prime\|_{\infty} \|\phi\|_2.
        \end{align*}
        Thus $K_m$ satisfies condition \eqref{eq:weakerHormander}. Finally, given $f \in \Ltwo(\Avnat)$ there holds
        \begin{align*}
            \|T_m(f) - T(f)\|_{\Ltwo} &\leq \|R_m T (R_m(f) - f)\|_{\Ltwo} + \|(R_m - \id)T(f)\|_{\Ltwo} \\
            &\leq \|T\| \ \|R_m(f) - f\|_{\Ltwo} + \|(R_m - \id)T(f)\|_{\Ltwo} \rightarrow 0.
        \end{align*}
        as $m$ grows to infinity.
  \end{proof}

    Before proving that a left Calder\'on-Zygmund operator extends to a bounded map from $\Hardyc$ into $\Lone(\mathcal{A})$, a fundamental property is included.

    \begin{proposition}
        Let $\mathcal{M}$ be a von Neumann algebra and let $T$ be a left Calder\'on-Zygmund operator which is bounded on $\Ltwo(\mathbb{R};\Ltwo(\mathcal{M}))$ whose kernel is zero. Then $T$ corresponds to the  left multiplication by some operator $F \in \Avnat$.
    \end{proposition}
    \begin{proof}
        Let $f \in \Avnat$ such that $\mbox{supp}\|f\|_2$ is compact and contained in some interval $J$. Then, we claim  that for any $g,g^\prime \in \mathcal S$ there holds
        \begin{align}\label{eq:CZpointwise}
            \langle g^\prime, T(gf) \rangle_{\Ltwo(\mathbb{R};\Ltwo(\mathcal{M}))} = \langle g^\prime, T(g) f \rangle_{\Ltwo(\mathbb{R};\Ltwo(\mathcal{M}))}.
        \end{align}
        As a consequence of Definition~\ref{def:CZ}, given $m \in \mathcal{M} \cap \Lone(\mathcal{M})$ and $I$ an arbitrary interval of positive length and $\lambda<1$, it holds
        \begin{align*}
            \langle g^\prime (\bigchi_{(\overline{I})^c} \otimes 1), T(g (\bigchi_{\lambda I} \otimes m)) \rangle &= \langle g(y) (\bigchi_{\lambda I} \otimes m)(y) g^\prime(x) (\bigchi_{(\overline{I})^c} \otimes 1)(x), K_{(\overline{I})^c,\lambda I} \rangle = 0, \\
            \langle g^\prime (\bigchi_{\lambda I} \otimes 1), T(g (\bigchi_{(\overline{I})^c} \otimes m)) \rangle &= \langle g(y) (\bigchi_{(\overline{I})^c} \otimes m)(y) g^\prime(x) (\bigchi_{\lambda I} \otimes 1)(x), K_{\lambda I,(\overline{I})^c} \rangle = 0.
        \end{align*}
        It follows that
        $$
        \langle g^\prime (\bigchi_{\lambda I} \otimes 1), T(g (\bigchi_{I^c} \otimes m)) \rangle =  \langle g^\prime (\bigchi_{\lambda I} \otimes 1), (\bigchi_{(\overline{I})^c} \otimes 1)T(g (1 \otimes m)) \rangle.$$
        Using modularity, we get
        $$ \langle g^\prime (\bigchi_{\lambda I} \otimes 1), T(g) (\bigchi_{I} \otimes m) \rangle = \langle g^\prime (\bigchi_{\lambda I} \otimes 1), T(g (\bigchi_{I} \otimes m)) \rangle.$$
        By continuity, this also holds for $\lambda=1$ and similarly
        \begin{align*}
            \langle g^\prime (\bigchi_{(\overline{I})^c} \otimes 1), T(g) (\bigchi_{I} \otimes m) \rangle &= 0 = \langle g^\prime (\bigchi_{(\overline{I})^c} \otimes 1), T(g (\bigchi_{I} \otimes m)) \rangle,
        \end{align*}
        yielding
        $\langle g^\prime, T(g)(\bigchi_{I} \otimes m) \rangle =
        \langle g^\prime, T(g(\bigchi_{I} \otimes m))
        \rangle$. By continuity of $T$, this can be extended to any $g,g'\in \Ltwo$. And finally, this identity also remains valid when replacing
        $\bigchi_{I} \otimes m$ by any $f$ by weak$^*$-density of elementary tensors (and Kaplansky's theorem) and by weak$^*$ continuity of all involved maps. Therefore, claim \eqref{eq:CZpointwise} follows.
           
        We have shown that $T$ commutes with right multiplications by $f\in  \Avnat$, thus it has to be a left multiplication by an element in  $\Avnat$ as the algebra  $\Avnat$ acts in standard form on $\Ltwo(\R;\Ltwo(\mathcal M))$. 
    \end{proof}

    \begin{theorem}\label{main}
        Let $\mathcal{M}$ be a von Neumann algebra. Let $T$ be a left Calder\'on-Zygmund operator with associated kernel $K : \mathbb{R} \times \mathbb{R} \setminus \{x = y\} \rightarrow \mathcal{M}$. If $K$ satisfies Hörmander condition \eqref{eq:hormanderCondition}, then $T$ extends to a bounded operator from $\Hardyc$ into $\Lone(\mathbb{R};\Lone(\mathcal{M}))$.
    \end{theorem}
    \begin{proof}
We can assume that the Hörmander condition is satisfied for some $\lambda\geq 1$.   First the operator $T$ can be defined on $c$-atoms. This is
      where we use the modularity. Indeed, given a $c$-atom with
      decomposition $a=bh$ supported on $I$, $T(b)h$ is well defined
      in $\Lone(\Avnat)$
      but it may depend on the
      decomposition. Let us justify that it does not.  First, for any
      $m\geq 1$, since $K_m(x,\cdot)$ belongs to
      $\Avnat$ (actually it is continuous), we have the formula
  \begin{align*}
            T_m(a)(x) = \int K_m(x,y) \ b(y) h \ dy = \int K_m(x,y) \ b(y) \ dt \cdot h=T_m(b)(x) \cdot h.
        \end{align*}    
         Consider another decomposition $a=b^\prime h^\prime$ supported on $J$. Then, we have $T_m(b)h= T_m(b^\prime)h^\prime$ in $\Lone(\Avnat)$. Thus, $T(a)$ can indeed be defined as $T(a)=T(b)h = T(b^\prime)h^\prime$. It can be justified as follows,  for any finite interval $K$, the norm of the difference
        \begin{align*}
            \|\bigchi_K\big( T(b)h - &T(b^\prime)h^\prime\big)\|_{\Lone(\Linfty(\mathbb{R}) \overline{\otimes} \mathcal{M})} \leq \|\bigchi_K\big(T(b)h - T_m(b)h\big)\|_{\Lone} + \|\bigchi_K\big(T_m(b^\prime)h^\prime - T(b^\prime)h^\prime\big)\|_{\Lone}\\
            &\leq \sqrt{|K|} \ \|h\|_{\Ltwo(\mathcal{M})} \ \|T(b)-T_m(b)\|_{\Ltwo} + \sqrt{|K|} \ \|h^\prime\|_{\Ltwo(\mathcal{M})} \ \|T(b^\prime)-T_m(b^\prime)\|_{\Ltwo} 
        \end{align*}
        goes to zero as $m\to \infty$  by Lemma~\ref{lem:approxKernel}. Hence, we indeed have $T(b)h = T(b^\prime)h^\prime$.

In particular $T$ and $T_m$ satisfy the assumptions of Lemma \ref{lem:CZatom} and we can conclude that for some $C' \geq 0$ independent of $m$ (due to \eqref{eq:weakerHormander}):
        \begin{align}\label{eq:absoluteBound}
            \|T(a)-T_m(a)\|_{\Lone(\mathbb{R};\Lone(\mathcal{M}))} \leq  C'.
        \end{align}
  Moreover, $R_m a = \phi_m * a$ and $\phi_m*a-a$ are multiples of  $c$-atoms by Lemma \ref{lem:convolutionAtoms} and  $\|\phi_m*a-a\|_{\Hardyc}\to 0$ as $m\to \infty$. Thus
        \begin{align} \nonumber
            \|T(a)-T_m(a)\|_{\Lone} &\leq \|R_m T R_m(a)-R_m T (a)\|_{\Lone} + \|R_m T(a) - T(a)\|_{\Lone} \\  \label{eq:CZestimate}
            &\leq \|T(R_m a - a)\|_{\Lone} + \|(R_m - \id)Ta\|_{\Lone},
        \end{align}
        which goes to $0$ with $m\to \infty$.

        In order to define $T$ on $\Hardyc$, assume $f$ is decomposed as
        $f=\sum_{i=1}^\infty \lambda_i a_i$ where $a_i$'s are
        $c$-atoms and $(\lambda_i)\in \ell_1$. Then, we check that
    $$T(f)=\sum_{i=1}^\infty\lambda_i T(a_i)\in \Lone(\Linfty(\mathbb{R}) \overline{\otimes} \mathcal{M})$$
    does not depend on the decomposition. Indeed, assume that $f=0$. Then, the extension above is well defined and
        \begin{align*}
            \sum_{i=1}^\infty \lambda_i T_m(a_i)(x) = \int K_m(x,y) \ \sum_{i=1}^\infty \lambda_i a_i(y) \ dy = 0
        \end{align*}
    holds since $K_m(x,\cdot)$ belongs to
    $\Avnat$; actually $T_m$ is bounded from $\Lone(\Linfty(\mathbb{R}) \overline{\otimes} \mathcal{M})$ to $\Linfty(\mathbb{R}) \overline{\otimes} \mathcal{M}$. Therefore, we obtain
        \begin{align*}
            \Big\| \sum_{i=1}^\infty \lambda_i \ T(a_i)\Big\|_{\Lone(\mathbb{R};\Lone(\mathcal{M}))} 
            &\leq  \Big\| \sum_{i=1}^\infty \lambda_i \ (T(a_i) -T_m(a_i)) \ \Big\|_{\Lone(\mathbb{R};\Lone(\mathcal{M}))} \\
            &\leq \sum_{i=1}^\infty |\lambda_i| \ \|T(a_i) - T_m(a_i)\|_{\Lone(\mathbb{R};\Lone(\mathcal{M}))} \\
            &\leq \sum_{i=1}^\infty |\lambda_i| \ \|T(a_i) - T_m(a_i)\|_{\Lone(\mathbb{R};\Lone(\mathcal{M}))}.
        \end{align*}
        The former series tends to $0$ as $m\to \infty$ as consequence of \eqref{eq:absoluteBound} and \eqref{eq:CZestimate}. The boundedness of $T$ is then clear as a consequence of Lemma~\ref{lem:CZatom}.
    \end{proof}

\begin{remark}\label{smallext}
We chose to write the results for Calder\'on-Zygmund defined from $L_2(L_\infty(\R)\overline \otimes\mathcal M)$ into itself. Actually, they extend without any modification to maps from $L_2(L_\infty(\R)\overline \otimes\mathcal M)$ to $L_2(L_\infty(\R)\overline \otimes\widehat {\mathcal M})$ with $\widehat{\mathcal M}$-valued kernels as long as $(\mathcal M,\tau)\subset (\widehat{\mathcal M},\widehat \tau)$ is an inclusion of von Neumann algebras with $\widehat \tau_{|\mathcal M}=\tau$ and $T$ has the right modular property with respect to $\mathcal M$.
\end{remark}

\begin{remark}
  As we already mentioned in the Introduction, all of our results hold
  for functions in $\mathbb{R}^n$. They also extend to
  more general measure metric spaces so long as the underlying measure
  is doubling, that is, if the condition
$$
\mu(B(x,2r)) \leq C_\mu \mu(B(x,r)),
$$
holds for all $x \in \mathrm{supp}(\mu)$ and all $r>0$ and that one has suitable approximations (such as in Lemma \ref{lem:garnett} or \ref{lem:approxKernel}). When $\mu$ is
a measure on $\mathbb{R}^n$ which fails the doubling condition the
definition of the appropriate $\BMO$-type space and a predual is more
involved and due to Tolsa in the classical case \cite{T01}. A
semicommutative definition in that context can be found in
\cite{CP19}. We let the interested reader to try to carry on all our
arguments to that setting.
\end{remark}

\section{An  example}\label{example}

We sketch how to use Theorem \ref{main} to
recover some results in \cite{M07}. A similar approach was used in
\cite{XXX16} to obtain characterizations of some Hardy spaces relying on
column-Hilbert spaces Calder\'on-Zygmund kernels, and indeed constitutes a particular case of our work.

Let $(\mathcal N,\tau)$ be a semifinite von Neumann algebra and consider
$$\mathcal M=\mathcal N\overline \otimes B(\Ltwo(\R^+,t \ dt)^2)$$ with its natural tensor product trace. We fix $\mathds{1}$ a function of norm 1 in $\Ltwo(\R^+,t \ dt)^2$.

Let $P(x,y)=\frac 1 \pi \frac y{x^2+y^2}$ be the Poisson kernel for the
upper half plane. We can define a continuous convolution kernel
$K: \R\times\R \ \setminus \ \{x=y\} \to  \mathcal M=\mathcal N\overline \otimes B(\Ltwo(\R^+,t \ dt)^2)$ by $K(x,v)=\mathcal K (x-v)$ where $\mathcal K$ is defined as 
$$\mathcal K(x)= 1_{\mathcal N} \otimes \Big(\big(\frac {\partial }{\partial x} P(x,\cdot), \frac {\partial }{\partial y} P(x,\cdot)\big)\otimes \mathds{1}\Big)$$
for $x\neq 0$. The map $\mathcal K$ is continuous and bounded on each closed interval not
containing 0. Since $\mathcal N$ does not play any role, $\mathcal K$
actually takes values in a column Hilbert space. It is a standard fact that it satisfies the Mikhlin condition. All this implies that $K$ satisfies the H\"ormander condition.

The  Calder\'on-Zygmund operator $T$ associated to $K$ can be defined. At
the $\Ltwo$-level, it acts only on a column subspace of the $S_2(\Ltwo(\R^+,t \ dt)^2)$ component. Thus, by homogeneity of Hilbertian operator spaces, we deduce
that it is bounded from the classical result with values in Hilbert spaces (i.e. going from $L_2(\mathbb R)$ to $L_2(\mathbb R; L_2(\mathbb R^+,t \ dt))$). Since $T$ is right-modular with respect to
$\mathcal N$ inside the tensor algebra $\mathcal N\overline \otimes B(\Ltwo(\R^+,t \ dt)^2)$
(with $n\mapsto n\otimes p_{\mathds 1}$ which preserves the traces), we indeed have a left Calder\'on-Zygmund operator from $L_\infty(\mathbb R)\overline \otimes \mathcal N$ to $L_\infty(\mathbb R)\overline \otimes \mathcal M$ by Remark \ref{smallext}.

Let $f$ be a compactly supported simple function on $\R$ with values
in $\mathcal N\cap \Lone(\mathcal N)$. If $F$
denotes the harmonic extension of $f$ to the upper half plane, the
column Littlewood-Paley function associated to $f$ is
$$G_c(f)(x)=\Big(\int_{\R^+} |\nabla F(x,y)|^2 y \ dy\Big)^{1/2}.$$
We can see $f$ as a function
with values in $\mathcal M$ by considering
$f\otimes(\mathds 1\otimes \mathds 1)$ as we already did. Some easy computations yield
$$\big|T (f\otimes p_{ \mathds 1})\big|^2=G_c(f)^2  \otimes p_{ \mathds 1}.$$
After giving suitable definitions for all the objects,  Theorem \ref{main} yields that there is a constant $C$ so that there holds 
\begin{equation}\label{meieq}\|G_c(f)\|_{\Lone(\Linfty(\R)\overline \otimes\mathcal M)}\leq C \|f\|_{\Honec(L_\infty(\mathbb R)\overline \otimes \mathcal N)}.\end{equation}
for every $f\in \Honec(\mathcal{A})$. In the same way, one can also get \eqref{meieq} with the Lusin square function $S_c(f)$ instead of the Littlewood-Paley one.

Consequently our definition for $\Hardyc$  matches that of \cite{M07}: we have just proved an inclusion and the dual spaces are the same (see Theorem 2.4 (i) and Corollary 2.7 there). 

As another direct application, we recover a dual result of \cite{JMP14} and close to \cite{HLMP14}: for scalar
kernels satisfying the H\"ormander condition, the associated
Calder\'on-Zygmund operators are bounded on
$L_p(\mathbb R;L_p(\mathcal M))$ when $1<p<2$.

\appendix
\section{Vector-valued Hardy spaces}\label{app:vectorhardy}

    This appendix contains an explicit argument for the construction of the map
    \begin{align*}
        Q: \Ltwozero(\mathbb{R},(1+t^2)dt; \Ltwo(\mathcal M)) \longrightarrow \Hone(\mathbb{R};\Ltwo(\mathcal{M}))
    \end{align*}
    as stated along the discussion preceding Proposition~\ref{prop:meyerMap}. Actually, the map $Q$ can be considered whenever $\Ltwo(\mathcal{M})$ is replaced by any other Banach space. Moreover, a brief study of \emph{molecules}, which give an additional description of the vector-valued Hardy space, is included.
    
    Given a Banach space $\mathbb{X}$ we say that a function $a$ belonging to $\Lone(\R;\mathbb{X})$ is a $\Ltwo(\R;\mathbb{X})$-atom in $\Hone(\R;\mathbb{X})$ whenever it satisfies the following conditions:
    \begin{itemize}
        \item $\mathrm{supp}(a) \subseteq I$ for some interval  $I$,
 \item $\int_{I} a = 0$.
        \item $\|a\|_{\Ltwo(\R;\mathbb{X})} \leq \frac{1}{\sqrt{|I|}}$,
       
    \end{itemize}
    Then, $\Hone(\R;\mathbb{X})$ is defined as the subspace of those functions $f$ in $\Lone$ admitting a decomposition
    \begin{align}\label{eq:vectorH1def}
        f = \sum\limits_{i=1}^\infty \lambda_i  a_i \quad \mathrm{in \ } \Lone(\R;\mathbb{X})
    \end{align}
    for some absolutely summable sequence $(\lambda_i)_{i=1}^\infty$
    and some family of $\Ltwo(\R;\mathbb{X})$-atoms
    $(a_i)_{i=1}^{\infty}$. According to \cite{H06}, the definition
    of $\Hone(\R;\mathbb{X})$ can be given via maximal functions the
    same way it is done in the scalar-valued case. This justifies
    which sort of convergence is considered at
    \eqref{eq:vectorH1def}. It can be checked that the norm
    \begin{align*}
        \|f\|_{\Hone(\R;\mathbb{X})} = \inf\{ \sum_{i=1}^{\infty} |\lambda_i| \ : \ f = \sum_{i=1}^{\infty} \lambda_i a_i \ \mathrm{for \ some \ } (\lambda_i)_i \in \ell_1 \ and \ \Ltwo\textrm{-atoms} \ (a_i)_i \}
    \end{align*}
    is equivalent to the one defined via maximal functions with values in arbitrary Banach spaces. On the other hand, we include below the characterization via \emph{molecules} of $\Hone(\R;\mathbb{X})$ which streamlines the proof for the classical case \cite[Ch. 5 Sec. 5]{Mey90}.

    Let $\mathbb{\mathbb{X}}$ be a Banach space and consider as before
    the function $\omega(x) = 1+x^2$. Then, the space of Bochner measurable functions
    \begin{align*}
        \Ltwo(\R,\omega \ dx; \mathbb{\mathbb{X}}) = \big\{ f \in L_0(\R;\mathbb{\mathbb{X}}) \ : \ \int_{\R} \|f(x)\|^2_\mathbb{\mathbb{X}} \ \omega(x) \ dx \ < \infty \big\}
    \end{align*}
    is contained in $\Lone(\R;\mathbb{\mathbb{X}})$. We will consider the subspace
    \begin{align*}
        M(\mathbb{\mathbb{X}}) = \big\{f \in \Ltwo(\R, \omega \ dx ; \mathbb{\mathbb{X}}) \ : \ \int_{\R} f = 0 \big\}.
    \end{align*}

    \begin{lemma}
        $M(\mathbb{\mathbb{X}})$ is a dense linear subspace of $\Hone(\R;\mathbb{\mathbb{X}})$.
    \end{lemma}
    \begin{proof}
     For $j\geq 0$, set $J_j=\{2^{j-1} < |x| \leq 2^j\}$ and $J_0=\{|x| \leq 1\}$. Let $f \in M(\mathbb{X})$ and define for $j\geq 0$, $f=f\bigchi_{J_j}$.
      
        These functions satisfy
        \begin{align*}
            \|f_j \|_{\Ltwo(\R;\mathbb{\mathbb{X}})} &= \Big( \int_{J_j} \|f(x)\|_\mathbb{\mathbb{X}}^2 \ dx \Big)^{1/2} \\
                       &\leq \Big( \int_{J_j} \|f(x)\|^2_\mathbb{X} \ \omega(x) \ dx \Big)^{1/2} \ 2^{-(j-1)}  =: R_j \ 2^{-j}
        \end{align*}
        so that the sequence $(R_j)_{j \geq 0}$ belongs to $\ell_2$. Let $I_j$ be the integral $\int_{\R} f_j$. Then, by the Cauchy-Schwarz inequality and a similar computation, it follows that, for $j \geq 0$
        \begin{align*}
            \| I_j \|_{\mathbb{X}} & \leq \Big( \int_{J_j} \|f(x)\|^2_\mathbb{X} \ \omega(x) \ dx \Big)^{1/2} \ \Big( \int_{J_j} \omega(x)^{-1} \ dx \Big)^{1/2} \\
            &\leq R_j \ 2^{-j/2}.
        \end{align*}
        Therefore, this yields some estimates for $S_j = \sum_{k \geq j} I_j$. Indeed,
        \begin{align*}
            \| S_j\|_\mathbb{X} \leq c \suma{k \geq j} R_k \ 2^{-k/2}.
        \end{align*}
        Now, let's replace the functions $f_j$ by some \emph{perturbed} atoms $a_j$ given by
        \begin{align*}
            a_j(x) = f_j(x) + S_{j+1} \ 2^{-(j+1)} \ \bigchi_{J_{j+1}}(x) - S_j \ 2^{-j} \ \bigchi_{J_j}(x).
        \end{align*}
        Then, the sequence $(a_j)_{j \geq 0}$ satisfies $\sum a_j=f$ as $S_0=0$ and 
        \begin{align*}
            \int_{\R} a_j(x) \ dx &= \int_{J_j} f(x) \ dx + S_{j+1} - S_{j} \\
            &= \int_{J_j} f(x) \ dx - I_j = 0.
        \end{align*}
        Moreover, it is easy to check that, by hypothesis, the support of $a_j$ is contained in $[-2^{j+1},2^{j+1}]$ and using the triangle inequality,
        \begin{align*}
            \|a_j\|_{\Ltwo(\R;\mathbb{X})}
            &\leq R_j \ 2^{-j} + 2c \suma{k \geq j} R_k \ 2^{-k/2} 2^{-j/2} =\lambda_j.
        \end{align*}
        Then $(\lambda_i)_{i \in \N} \in \ell_1$ by previous computations and since
        \begin{align*}
          \sum_{j=0}^\infty \sum_{k \geq j} 2^{-j/2} R_j \ 2^{-k/2} &\lesssim \sum_{j=0}^\infty R_j 2^{-j} \lesssim   \Big(\sum_{j=0}^\infty R_j^2\Big)^{1/2} =2                                                                        \|f\|_{\Ltwo(\mathbb{R},\omega dt;\Ltwo(\mathcal{M}))}.
        \end{align*}
        Moreover, redefining $a_i$ as $a_i / \lambda_i$, we obtain the expression
        \begin{align*}
            f = \sum_{i\in\N} \lambda_i \ a_i
        \end{align*}
        where each $a_i$ is an atom with support $[-2^{j+1},2^{j+1}]$ and convergence holds in the sense of $\Lone(\mathbb{R};\mathbb{X})$. In conclusion, $f$ belongs to $\Hone(\R;\mathbb{X})$.
    \end{proof}

    \begin{definition}
         A \emph{molecule} $f$ with values in $\mathbb X$, centered at $x_0$ and of width $d > 0$, is defined to be a function belonging to $M(\mathbb{X})$ which is normalized by
        \begin{align*}
            \Big( \int_{\R} \|f(x)\|_\mathbb{X}^2 \ \Big(1 + \frac{|x-x_0|^2}{d^2}\Big) \ dx \Big)^{1/2} \leq d^{-1/2}.
        \end{align*}
    \end{definition}

    The $\Hone(\R;\mathbb{X})$-norm is invariant by translations and
    homogeneous for composition with homothethies. It follows that a
    molecule $f$ sits in $\Hone(\R;\mathbb{X})-$ with a norm
    controlled by an absolute constant. One can check that an atom is a
    molecule. Thus, one can use molecules instead of atoms in the
    definition of $\Hone(\R;\mathbb{X})$.    

    \section{Proof of Lemma~\ref{lem:garnett}}\label{app:garnett}

    The justification for Lemma~\ref{lem:garnett} is included in this
    appendix. This result was stated in the commutative setting by
    Garnett \cite{G07}, and also by Mei without proof \cite{M07}. For
    that reason, a general version of the argument is developed
    here. Before giving the explicit construction, some preliminary
    classical results are showed. We will use the easier form
    \eqref{eq:simple} and to lighten the notation we drop $\otimes \mathds 1$.
\begin{lemma}\label{lem:auxiliarGarnett1}
  Suppose $f \in \Linfty(\mathcal{M};\Ltwoc(\mathbb{R},\frac{dt}{1+t^2}))$ and let $I\subset J$ be finite intervals. Then, there holds
  $$\big\| \frac1{\sqrt {|I|}} \widetilde{\iota}_I(f -f_I) \big\|_{\Linfty(\mathcal{M};\Ltwoc(\mathbb{R}))} \leq\sqrt{ \frac {|J|}{|I|}}  \big\| \frac1{\sqrt {|J|}} \widetilde{\iota}_J(f - f_J) \big\|_{\Linfty(\mathcal{M};\Ltwoc(\mathbb{R}))}.$$
   \end{lemma}
    \begin{proof} This is clear because one goes from $\iota_J(f - f_J)$
just by restricting and centering with respect to $I$ and all these operations are contractions at the $L_2$-level.
            \end{proof}

    Let $J = (a,b]$ be a finite interval and let $c_J = \frac{a+b}{2}$ denote the center of $J$. Write 
    \begin{align*}
        J = J_0 \cup \bigcup_{n=1}^\infty J_n \cup \bigcup_{n=1}^\infty J_n^\prime
    \end{align*}
    where $d(J_n, \partial J) = |J_n|$ for $n \geq 0$ and $d(J_n^\prime,\partial J) = |J_n^\prime|$ for $n \geq 1$. Then $J_0$ coincides with the middle third of $J$, that is,
    \begin{align*}
        J_0 = \frac{1}{3} J = \big( c_J - \frac{1}{2} \frac{1}{3} |J|, c_J + \frac{1}{2} \frac{1}{3} |J|\big],
    \end{align*}
    while for any $n \geq 1$,
    \begin{align*}
        J_n &= \big(c_J + \frac{|J|}{3} \sum_{k=0}^{n-1} \frac{1}{2^k}, c_J + \frac{|J|}{3} \sum_{k=0}^{n} \frac{1}{2^k} \big] \\
        J_n^\prime &= \big( c_J - \frac{|J|}{3} \sum_{k=0}^{n} \frac{1}{2^k} , c_J - \frac{|J|}{3} \sum_{k=0}^{n-1} \frac{1}{2^k} \big].
    \end{align*}
    Thus, $|J_n| = |J_n^\prime| = \frac{|J|}{3}\frac{1}{2^n}$. Set $K_n$ (respectively $K_n^\prime$) as the reflection of $J_n$ (respectively $J_n^\prime$) across $b$ (respectively $a$). Then,
    \begin{align*}
        K_n &= \big[ b + \frac{|J|}{3} - \frac{|J|}{3} \sum_{k=1}^n \frac{1}{2^k}, b + \frac{|J|}{3} - \frac{|J|}{3} \sum_{k=1}^{n-1} \frac{1}{2^k} \big), \\
        K_n^\prime &= \big[ a - \frac{|J|}{3} + \frac{|J|}{3} \sum_{k=1}^{n-1} \frac{1}{2^k}, a - \frac{|J|}{3} + \frac{|J|}{3} \sum_{k=1}^{n} \frac{1}{2^k} \big), \\
    \end{align*}
    so $|K_n| = |K_n^\prime| = |J_n| = |J_n^\prime|$. Finally, define $L = [b + \frac{|J|}{3}, \infty)$ and $L^\prime = (-\infty, a - \frac{|J|}{3})$. This construction yields the desired operator $\psi$, which is given by the following expression assuming the normalization $\varphi_J=0$ (recall that $\varphi$ is defined up to a constant factor):
    \begin{align*}
        \psi &= \widetilde{P}_J \varphi + \sum_{n \geq 1} \varphi_{J_n} \otimes \bigchi_{K_n}  + \sum_{n \geq 1} \varphi_{J_n^\prime} \otimes \bigchi_{K_n^\prime}.  
    \end{align*}
    Let $J^+=\cap_{n\geq 1} J_n$ and $J^-=\cap_{n\leq 1}J_n$. It is worth
    to mention that if $\mathbb E^+$ denotes the conditional
    expectation on $L_2(J^+)$ corresponding to the partition
    $\{J_n; n\geq1\}$, then
    $\sum_{n \geq 1} \varphi_{J_n} \otimes \bigchi_{K_n}$ is nothing but $
    \widetilde R_+ (\widetilde {\mathbb E^+}) (\widetilde{P}_{J^+} \varphi)$ where $R_+$ is the reflection with respect to $b$. A similar formula holds for $\sum_{n \geq 1} \varphi_{J_n^\prime} \otimes \bigchi_{K_n^\prime}$.

    We turn to the proof of Lemma~\ref{lem:garnett}. We need to bound, for all finite intervals $I$, the quantity $$\big\| \frac1{\sqrt {|I|}} \iota_I(\psi -\psi_I) \big\|_{\Linfty(\mathcal{M};\Ltwoc(\mathbb{R}))}.$$ 

    We distinguish according to $I$:

    {\it Case 1}: $I\subset J$; there is nothing to do
    as $\iota_{I}(\psi -\psi_{I})=\iota_{I}(\varphi -\varphi_{I})$.
    
    {\it Case 2}: $b\in I$.

    Using Lemma \eqref{lem:auxiliarGarnett1}, we can assume that $I$ is symmetric around $b$ losing a factor $\sqrt 2$.

    If $a\notin I$, then $I\subset \cup_{n=n_0}^\infty(J_n\cup
    K_n)$ with $J_{n_0-1}\cap I=\emptyset$. Similarly, by enlarging it and losing another factor $\sqrt{2}$,     we can assume it is has the form $\cup_{n=n_0}^\infty(J_n\cup  K_n)$ because $|J_{n_0}|\leq |I|$.  Then $\psi_I=\psi_{I^+}=\psi_{I^-}=\varphi_{I^-}$ where $I^-=\cup_{n=n_0}^\infty J_n$ and $I^+=\cup_{n=n_0}^\infty K_n$, $|I|=2|I^-|$. Note that $\widetilde R_+$ yields an isometry and
    $$\| \frac1{\sqrt {|I^+|}} \iota_{I^+}(\psi -\psi_{I^+}) \big\|_{\Linfty(\mathcal{M};\Ltwoc(\mathbb{R}))}=\| \frac1{\sqrt {|I^-|}} \widetilde {\mathbb E^+}\iota_{I^-}(\varphi -\varphi_{I^-}) \big\|_{\Linfty(\mathcal{M};\Ltwoc(\mathbb{R}))}.$$    Thus using that $\widetilde {\mathbb E^+}$ is a contraction commuting with $\iota_{I^-}$, we get that 
    $$\| \frac1{\sqrt {|I|}} \iota_I(\psi -\psi_I) \big\|_{\Linfty(\mathcal{M};\Ltwoc(\mathbb{R}))}^2\leq \| \frac1{\sqrt {|I^+|}} \iota_{I^+}(\psi -\psi_{I^+}) \big\|_{\Linfty(\mathcal{M};\Ltwoc(\mathbb{R}))}^2\leq \|\varphi\|_{\BMOc}.$$

    If $a\in I$, then we can simply replace it with $3J$ which has a comparable size. then a similarly reasoning as above also gives a control by $\| \frac1{\sqrt {|J|}} \iota_{J}(\varphi -\varphi_{J}) \big\|_{\Linfty(\mathcal{M};\Ltwoc(\mathbb{R}))}$.
    
    {\it Case 3}: $a\in I$. One can proceed as in case 1.

    {\it Case 4}: $I\subset (b,\infty)$.

    If $I\supset K_n$ for some $n$, then $|I|\geq |K_n|$ and losing a factor $\sqrt 2$ we can replace it by the bigger interval $\cup_{k\geq n} K_k\cup I$ which has size less than $2|I|$. We are then back to case 2.

    If $I\subset K_n$ for some $n$ there is noting to do as $\iota_I(\psi -\psi_I)=0$.

    If $I\subset K_n\cup K_{n+1}$ for some $n\geq 1$. By losing a constant factor we can assume that $I$ is symmetric with respect to left border of $K_n$ say $b_n$.
    If the size of $I$ is smaller than $2|K_{n+1}|$, then $\iota_I(\psi -\psi_I) $ is a simple function taking two values $\pm\frac 1 2  (\psi_{K_n}-\psi_{K_{n+1}})$. Then the norm of $\frac 1 {\sqrt{|I|}}\iota_I(\psi -\psi_I)$ is
    $\|\frac 1 2  (\psi_{K_n}-\psi_{K_{n+1}})\|$. It the same as if we consider
    $(b_n-|K_{n+1}|,b_n+|K_{n+1}|)$ and we are back to situation where $K_{n+1}\subset I$ (or we can use the argument of case 2).

    The only remaining situation is when $I\subset K_1\cup L$. Then an easy calculation gives that the involved norm is controlled
    by $\|\sqrt{2} \psi_{K_1}\|=\|\sqrt{2} \varphi_{J_1}\|$. This is where we use that $\varphi_J=0$ to say that (recall $6|J_1|=|J|$)
    $$\|\varphi_{J_1}\|\leq \sqrt 6 \| \frac 1{\sqrt{| J|}}
    \widetilde {\mathbb E^+} i_J (\varphi -0)\|\leq \sqrt{6} \|\varphi\|_{\BMOc}.$$ 

    {\it Case 5}: $I\subset (-\infty,a)$. It can be done similarly.

    We have dealt with all possible situations for $I$. This conclude Lemma~\ref{lem:garnett} with at most $C=2\sqrt 6$.

\subsection*{Acknowledgements}
The first author was supported by Grant SEV-2015-0554-19-3 funded by MCIN/AEI/10.13039/501100011033 and by ESF Investing in your future. In addition, he was partially supported by Grant PID 2019-107914GB-I00 (MCIN PI J. Parcet), and by the Madrid Government Program V under PRICIT Grant SI1/PJI/2019-00514. The second author was supported by ANR-19-CE40-0002.

    \bibliographystyle{amsplain}
    \bibliography{biblio}

\end{document}